\documentclass[11pt,reqno,oneside]{amsart}
\usepackage[utf8]{inputenc}
\usepackage[english]{babel}
\usepackage[babel]{csquotes}
\usepackage[backend=bibtex, hyperref]{biblatex}
\bibliography{MPGauss}

\usepackage{amsmath}
\usepackage{amsfonts}
\usepackage{amssymb}
\usepackage{amsthm} 
\usepackage{dsfont} 
\usepackage[mathscr]{euscript} 
\usepackage{eurosym}
\usepackage{bbm} 
\usepackage[a4paper]{geometry}
\usepackage{comment} 
\usepackage{paralist} 

\usepackage{graphicx}
\usepackage{xkeyval}
\usepackage{pstricks}
\usepackage{hyperref}
\usepackage{longtable}
\usepackage{tikz}

\hypersetup{
urlcolor=black, 
  menucolor=black, 
  citecolor=black, 
  anchorcolor=black, 
  filecolor=black, 
  linkcolor=black, 
  colorlinks=true,
}
\usepackage{comment}

\usetikzlibrary{matrix}

\newcommand{\Prob}{\mathds{P}}

\newcommand{\E}{\mathds{E}}
\newcommand{\V}{\mathds{V}}

\newcommand{\R}{\mathbb{R}}

\newcommand{\N}{\mathbb{N}}

\newcommand{\abs}[1]{|{#1}|} 
\newcommand{\bigabs}[1]{\left|{#1}\right|} 
\newcommand{\one}{\mathds{1}}

\newcommand{\Bcal}{\mathcal{B}}

\newcommand{\Pcal}{\mathcal{P}}
\newcommand{\Gcal}{\mathcal{G}}

\newcommand{\Ncal}{\mathcal{N}}

\newcommand{\Tcal}{\mathcal{T}}

\newcommand*{\defeq}{\mathrel{\vcenter{\baselineskip0.5ex \lineskiplimit0pt
                     \hbox{\scriptsize.}\hbox{\scriptsize.}}}%
                     =}
										
\newcommand{\integrala}[2]{\left\langle{#1},\,{#2}\right\rangle}

\newcommand{\de}{\text{d}}

\newcommand{\opnorm}[1]{\|#1\|_{\mathrm{op}}}

\newcommand{\ubar}[1]{\underline{#1}}
\newcommand{\oneto}[1]{[{#1}]}

\DeclareMathOperator{\Cov}{Cov}

\DeclareMathOperator{\tr}{tr}

\newcommand{\PePe}[1]{\mathcal{PP}(#1)}

\theoremstyle{plain}
\newtheorem{lemma}{Lemma}
\newtheorem{theorem}[lemma]{Theorem}

\theoremstyle{definition}

\theoremstyle{remark}

\newcommand{\1}{\one}
\definecolor{darkblue}{rgb}{.1, 0.1,.8}
\definecolor{darkgreen}{rgb}{0,0.8,0.2}
\definecolor{darkred}{rgb}{.8, .1,.1}

\newcommand{\nto}{n \to \infty}

\newcommand{\as}{{\rm a.s.}}

\usepackage{mathtools}
\mathtoolsset{showonlyrefs}

\newcommand{\cas}{\stackrel{\rm a.s.}{\rightarrow}}
\newcommand{\eid}{\stackrel{d}{=}}

\begin{document}
\title[Gaussian Sample Covariance Matrices]
{Large Sample Covariance Matrices of Gaussian Observations with Uniform Correlation Decay}
\author[Michael Fleermann]{Michael Fleermann}
\author[Johannes Heiny]{Johannes Heiny}
\begin{abstract}
We derive the Marchenko-Pastur (MP) law for sample covariance matrices of the form $V_n=\frac{1}{n}XX^T$, where $X$ is a $p\times n$ data matrix and $p/n\to y\in(0,\infty)$ as $n,p \to \infty$. We assume the data in $X$ stems from a correlated joint normal distribution. In particular, the correlation acts both across rows and across columns of $X$, and we do not assume a specific correlation structure, such as separable dependencies. Instead, we assume that correlations converge uniformly to zero at a speed of $a_n/n$, where $a_n$ may grow mildly to infinity. We employ the method of moments tightly: We identify the exact condition on the growth of $a_n$ which will guarantee that the moments of the empirical spectral distributions (ESDs) converge to the MP moments. If the condition is not met, we can construct an ensemble for which all but finitely many moments of the ESDs diverge. We also investigate the operator norm of $V_n$ under a uniform correlation bound of $C/n^{\delta}$, where $C,\delta>0$ are fixed, and observe a phase transition at $\delta=1$. In particular, convergence of the operator norm to the maximum of the support of the MP distribution can only be guaranteed if $\delta>1$. The analysis leads to an example for which the MP law holds almost surely, but the operator norm remains stochastic in the limit, and we provide its exact limiting distribution.
\end{abstract}
\keywords{Sample covariance matrices, Marchenko-Pastur law, correlated Gaussian, operator norm}
\subjclass[2020]{Primary: 60B20. Secondary: 60F05, 60G55} 
\maketitle

\section{Introduction}

In contemporary statistical analyses, data sets are typically large both in the sample size and in the dimension of the observations.  Since many test statistics are based on the eigenvalues of the sample covariance matrix, their asymptotic study has gained wide popularity and led to applications in many fields of modern sciences. In the classical setting of data matrices with independent and identically distributed (i.i.d.)\ entries, the famous Marchenko-Pastur law \cite{marchenko:pastur:1967} describes the asymptotic distriution of these eigenvalues.

Ever since, the universality of the Marchenko-Pastur law is under close investigation: How far may we deviate from the classical i.i.d.\ assumption and still obtain the Marchenko-Pastur distribution  as a limit? For example, in practical applications it is plausible to assume that data may be correlated or at least dependent. Even up to today, results for matrices with dependencies spreading over all entries are rather sparse: The recent paper \cite{Bryson:Vershynin:Zhao:2021}, for example, establishes the Marchenko-Pastur law for data matrices with independent observations (columns), but where covariates within these columns may exhibit stochastic dependencies (but are assumed to be uncorrelated). They also provide an interesting review of previous literature in this direction. 

Many previous analyses (cf. \cite{Bryson:Vershynin:Zhao:2021, Yaskov:2016, bai:zhou:2008}) investigate the case where rows (resp.\ columns) of the data matrices are assumed to be independent and stochastic dependencies may only prevail within these rows (resp.\ columns). 
Studies which allow correlation to span both across rows and columns are sparse: In \cite{paul:silverstein:2009}, separable sample covariance matrices were studied, which are of the form $n^{-1}X_n X_n^T$ with $X_n=\sqrt{A_n}Y_n\sqrt{B_n}$, where $A_n$ is a $p\times p$ deterministic Hermitian and positive semidefinite matrix, $B_n$ is an $n\times n$ deterministic diagonal matrix with non-negative real-valued entries, and $Y_n$ has i.i.d.\ entries with mean zero and unit variance. In this model the covariance structure of $X_n$ is given by the Kronecker product $A_n\otimes B_n$, which is a $pn\times pn$ matrix. Since $A_n\otimes B_n$ is determined by the $p^2+n^2$ entries of $A_n$ and $B_n$, covariance structures determined by Kronecker products are somewhat limited in their variability. In other words, some covariance structures -- such as the equicovariant structure (in Theorem~\ref{thm:rankoneexample} below) -- cannot be expressed as a Kronecker product.
Another study where correlations may span both across rows and columns is \cite{Fleermann:Heiny:2020}, where the authors derived the Marchenko-Pastur law for data matrices filled with jointly correlated random spins stemming from the Curie-Weiss model from statistical physics. There, even for the case of non-vanishing correlations, the Marchenko-Pastur law is recovered. 

The present paper studies the case where data is jointly Gaussian distributed and where correlations may act both across rows and columns of the data matrix. In Theorem~\ref{thm:MPlaw}, we do not impose any conditions on the correlation structure except that we require all correlations to decay uniformly at a speed of $a_n/n$, where $a_n=o(n^{\epsilon})$ for all $\epsilon>0$. We employ the method of moments to derive our results and we show that if not $[\forall\,\epsilon >0:\, a_n= o(n^{\epsilon})]$, the Marchenko-Pastur law cannot be guaranteed by the method of moments. More precisely, there exist ensembles matching the condition, such that all sufficiently large moments of a subsequence of the ESDs diverge to infinity, allowing no conclusion about weak convergence. In Theorem~\ref{thm:opnorm}, we investigate the convergence of the operator norms of our ensembles for which we assume that correlations decay uniformly at rates $C/n^{\delta}$, where $C$ and $\delta$ are fixed in $(0,\infty)$. We derive the following phase transition: For $\delta>1$, the operator norm of the sample covariance matrices converges almost surely to the right endpoint of the support of the Marchenko-Pastur distribution. For $\delta<1$, there exist ensembles matching the condition such that the operator norm converges to infinity. For $\delta=1$, we provide an ensemble for which the operator norm remains stochastic in the limit, so in particular, it does not converge to a constant, nor to infinity. Finally, the case where all correlations are of the same size will be studied in depth in Theorem~\ref{thm:rankoneexample}, where we provide precise limit results for the operator norm of the sample covariance matrices.

 The rest of this paper is organized as follows. In Section~\ref{sec:setupresults}, we introduce our setup and present the main results, Theorems~\ref{thm:MPlaw},\,~\ref{thm:opnorm}, and~\ref{thm:rankoneexample}. Section~\ref{sec:moments} is devoted to the proof of Theorem~\ref{thm:MPlaw}, whereas Section~\ref{sec:Eval} contains the proofs of Theorem~\ref{thm:opnorm} and Theorem~\ref{thm:rankoneexample}.

\section{Setup and Results}
\label{sec:setupresults}
We assume that observations $x_1,\ldots,x_n$ are $p$-dimensional random vectors, so that we obtain a $p\times n$ data matrix $X_n$ with columns $x_i$.
Based on these data, we define the sample covariance matrix of dimension $p\times p$,
\[
V_n = \frac{1}{n}\sum_{k=1}^n x_kx_k^T = \frac{1}{n} X_n X_n^T,
\]
and analyze the eigenvalues of $V_n$ as its dimensions tend to infinity. To be precise, we assume that the number of observations $n$ and the number of covariates $p$ grow asymptotically proportionally with each other, so that $p/n \to y\in(0,\infty)$ as $n,p\to \infty$. This assumption is typical in random matrix theory, see for instance the monographs \cite{BaiSi, yao:zheng:bai:2015}.
Throughout this paper, we assume that $p$ is a function of $n$, i.e., $p=p(n)$, but for simplicity we suppress this dependence notationally. In the following, we write $\oneto{k}\defeq \{1,\ldots,k\}$ for any $k\in\N$. 

 We impose the following assumptions about the sequence of data matrices $(X_n)_n$:
\begin{enumerate}
\item[(A1)] \label{eq:unitvar} For all $n\in\N$, $X_n\sim \Ncal(0,\Sigma_n)$, where $\Sigma_n$ is a positive semidefinite $pn\times pn$ matrix, indexed by pairs $(a,b)\in\oneto{p}\times\oneto{n}$, so that
\[
\Sigma_n = (\Sigma_n((a,b),(c,d)))_{(a,b),(c,d)\in\oneto{p}\times\oneto{n}}
\]
and $\Sigma_n((a,b),(a,b))=1$ for all $n\in\N$ and $(a,b)\in\oneto{p}\times\oneto{n}$.
\item[(A2)] \label{eq:corrdecay} We assume that there exists a sequence $(a_n)_n$ in $\R_+$ with $a_n = o(n^{\epsilon})$ for all $\epsilon>0$ such that the sequence $(\Sigma_n)_n	$ satisfies:
	\begin{equation}
	\label{eq:decay}
	\forall\ (a,b)\neq(c,d)\in\oneto{p}\times\oneto{n}:\ \abs{\E X_n(a,b)X_n(c,d)} = \abs{\Sigma_n((a,b),(c,d))} \leq \frac{a_n}{n}.
	\end{equation}
\item[(A3)] There is a constant $y\in(0,\infty)$ such that $p/n\to y$ as $n\to\infty$.
\end{enumerate}
Assumption (A1) states that all observations are jointly normal and have variance $1$. Assumption (A2) can be considered the main assumption in our model: We do not require correlations to follow a specific structure, but we do require them to decay uniformly at a speed of $a_n/n$. 

For all $n\in\N$, we define $\mu_n$ to be the empirical spectral distribution (ESD) of $V_n\defeq n^{-1}X_n X_n^T$, i.e.\
\[
\mu_n = \frac{1}{p} \sum_{i=1}^p \delta_{\lambda_i(V_n)}, 
\]
where $\lambda_1(V_n)\leq \ldots \leq \lambda_p(V_n)=\lambda_{\max}(V_n)=\opnorm{V_n}$ are the eigenvalues of $V_n$ and $\delta_x$ denotes the Dirac measure at point $x\in \R$. Note that the probability measures $\mu_n$ are random, as they depend on the random eigenvalues of $V_n$. 

Notationally, if $\nu$ is a probability measure on $(\R,\Bcal)$ and $f:\R\to\R$ is $\nu$-integrable, we write $
\integrala{\nu}{f} \defeq \int_{\R} f \de\nu$,
where when in doubt, $x$ is the variable of integration, for example, if $k\in\N$, then $
\integrala{\nu}{x^k} \defeq \int_{\R} x^k \de\nu(x)$. Now as $\mu_n$ is a random probability measure, its $k$-th moment $\integrala{\mu_n}{x^k}$ is a real-valued random variable.
 
Our first goal is to see $\mu_n$ converge weakly almost surely to the Marchenko-Pastur distribution $\mu^{y}$ with ratio index $y$, where 
\[
\mu^{y}(\de x) = \frac{1}{2\pi xy}\sqrt{(b-x)(x-a)}\one_{(y_-,y_+)}(x) \de x + \left(1-\frac{1}{y}\right)\one_{y>1}\delta_0 (\de x),
\]
$y_- = (1-\sqrt{y})^2$, $y_+=(1+\sqrt{y})^2$. The moments of $\mu^{y}$ are given by (see \cite{BaiSi}, for example)
\begin{equation}
\label{eq:MPmoments}	
\forall\, k\in\N: \integrala{\mu^y}{x^k} = \sum_{r=0}^{k-1} \frac{y^r}{r+1}\binom{k}{r}\binom{k-1}{r}.
\end{equation} 

\begin{theorem} \label{thm:MPlaw} The following statements hold:
\begin{enumerate}[i)]
\item Let $(X_n)_n$ be an ensemble of jointly Gaussian observations satisfying assumptions (A1), (A2) and (A3), and let $\mu_n$ be the ESD of $V_n\defeq n^{-1}X_nX_n^T$. Then the moments of $\mu_n$ converge almost surely to the moments of $\mu^y$, that is
\begin{equation}\label{eq:mom}
\forall\,k\in\N:\ \integrala{\mu_n}{x^k} \xrightarrow[n\to\infty]{} \integrala{\mu^y}{x^k} \quad \as
\end{equation}
In particular, \eqref{eq:mom} implies $\mu_n \to \mu^y$ almost surely as $n\to\infty$.
\item The result in $i)$ is tight in the following sense: For any sequence $(a_n)_n$ which does not satisfy $a_n = o(n^{\epsilon})$ for all $\epsilon>0$, there exists an ensemble $(X_n)_n$ satisfying assumptions (A1), (A3), and also \eqref{eq:decay}, such that all but finitely many moments of a subsequence of $\mu_n$ diverge to infinity.
\end{enumerate}
\end{theorem}
Part $ii)$ of Theorem \ref{thm:MPlaw} shows that our conditions are methodically tight, i.e., if the uniform correlation bound in (A2) is relaxed, then the method of moments cannot be used to infer weak convergence.

Our second goal of this paper is to investigate the behavior of $\opnorm{V_n}$, the operator norm of $V_n$. If the Marchenko-Pastur law holds, it is not unreasonable to expect $\opnorm{V_n}$ to converge almost surely to the right endpoint $y_+$ of the Marchenko-Pastur distribution. Part $i)$ in the following theorem shows that this is indeed the case for a large subclass of the models we investigated in Theorem~\ref{thm:MPlaw} $i)$. 
However, in part $iii)$ of the following theorem we see that convergence of $\opnorm{V_n}$ to $y_+$ need not take place, although Theorem~\ref{thm:MPlaw} $i)$ is applicable. Instead, the operator norm may remain stochastic in the limit.

\begin{theorem}
\label{thm:opnorm}
Let $(X_n)_n$ be an ensemble of jointly Gaussian observations satisfying assumptions (A1), (A3), and
\begin{equation}
\label{eq:deltacondition}
\forall\ (a,b)\neq(c,d)\in\oneto{p}\times\oneto{n}:\ \abs{\E X_n(a,b)X_n(c,d)} = \abs{\Sigma_n((a,b),(c,d))} \leq \frac{C}{n^{\delta}} \tag{A2'}
\end{equation}
for some fixed $C,\delta>0$. Then the following statements hold for $V_n = n^{-1}X_nX_n^T$.
\begin{enumerate}[i)]
\item If $\delta>1$, then $\opnorm{V_n}\to y_+=(1+\sqrt{y})^2$  almost surely as $\nto$. 
\item If $\delta\in(0,1)$, there exist ensembles satisfying the assumptions of the theorem such that $\opnorm{V_n}\to\infty$ almost surely as $\nto$.
\item If $\delta=1$, there exists an ensemble satisfying the assumptions of the theorem such that, as $\nto$, $\opnorm{V_n}$ converges almost surely to some non-degenerate random variable.
\end{enumerate}
\end{theorem}

Statements $ii)$ and $iii)$ of Theorem~\ref{thm:opnorm} require the construction concrete examples. The leading example (a variant of which is also used for the proof of Theorem~\ref{thm:MPlaw} $ii)$) is a setup we call \emph{equicovariant ensemble}: An ensemble $(X_n)_n$ of centered jointly Gaussian observations will be called \emph{equicovariant with sequence $(a_n)_n$}, if for all $n\in\N$ and  all $(i,j)\neq (k,\ell)\in\oneto{p}\times\oneto{n}$ we have $\E X_n(i,j) X_n(k,\ell)=a_n$. Statements $ii)$ and $iii)$ of Theorem~\ref{thm:opnorm} follow from the next result.

\begin{theorem}
\label{thm:rankoneexample}
Let $V_n = n^{-1} X_nX_n^T$, where $X_n$ satisfies (A1), (A3) and is equicovariant with sequence $a_n=n^{-\delta}$ for some fixed $\delta>0$. Then as $\nto$, the following statements hold:
\begin{enumerate}[i)]
\item If $\delta > 1$, then $\opnorm{V_n}$ converges almost surely to $y_+=(1+\sqrt{y})^2$.
\item If $\delta \in (0,1)$, then 
$$n^{\delta-1} \opnorm{V_n}\ \xrightarrow[n\to\infty]{\text{d}}\ yZ^2$$
where $Z$ is a standard normal random variable. 
\item If $\delta =1$, then
\begin{equation*}
\opnorm{V_n}\ \xrightarrow[n\to\infty]{\text{d}}\ \left(1+\sqrt{y}\right)^2 \1_{\{Z^2\le 1+ \sqrt{y^{-1}}  \}} + Z^2 \left(y+\frac{1}{Z^2-1}\right) \1_{\{Z^2> 1+ \sqrt{y^{-1}}  \}}\,,
\end{equation*}
where $Z$ is a standard normal random variable.
\end{enumerate}
Further, the ensemble $X_n$ can be constructed so that the convergence in $ii)$ and $iii)$ holds almost surely.
\end{theorem}

\section{Proof of Theorem~\ref{thm:MPlaw}: Analysis of moments.}
\label{sec:moments}

In order to prove Theorem~\ref{thm:MPlaw}, we employ the method of moments. To apply this method, it is sufficient to carry out the following two steps. First, we show that
\begin{equation}
\label{eq:expectation}
\forall\, k\in\N:\ \E\integrala{\mu_n}{x^k}\ \xrightarrow[n\to\infty]{}\ \integrala{\mu^{y}}{x^k}.
\end{equation}
Second, if for all $k\in\N$, $\integrala{\mu_n}{x^k}=D^{(k,1)}_n +\ldots + D^{(k,\ell_k)}_n$ is a finite decomposition such that for all $i\in\oneto{\ell_k}$, $\E D^{(k,i)}_n$ converges to a constant (this decomposition will become clear when showing \eqref{eq:expectation}), then we show that
\begin{equation}
\label{eq:variance} 
\forall\, k\in\N:\ \forall\, i\in\oneto{\ell_k}:\ \V D^{(k,i)}_n \xrightarrow[n\to\infty]{} 0 \quad\text{summably fast.} 
\end{equation}
Indeed, by virtue of \eqref{eq:expectation} we obtain weak convergence of $\mu_n$ to $\mu^y$ in expectation, and \eqref{eq:variance} yields almost sure convergence of the random moments which entails almost sure weak convergence of $\mu_n$ to $\mu^y$. See e.g.\ \cite{Fleermann:Kirsch:2022c} for details. This section is organized as follows: In Subsection 3.1 we introduce combinatorial concepts needed for our proof. in Subsection 3.2 we derive convergence of expected moments \eqref{eq:expectation}, and in Subsection 3.3 we show that the variances of the decomposed random moments decay summably fast \eqref{eq:variance}.

\subsection{Combinatorial Preparations}

To show \eqref{eq:expectation} and \eqref{eq:variance}, we need the moments of $\mu_n$ and $\mu^{y}$. The moments of $\mu^{y}$ are given above in \eqref{eq:MPmoments}, whereas we may calculate the moments of $\mu_n$ by
\begin{align}
&\integrala{\mu_n}{x^k}\ =\ \integrala{\frac{1}{p}\sum_{s=1}^p\delta_{\lambda_s}}{x^k}\ = \frac{1}{p}\sum_{s\in\oneto{p}}\lambda_s^k
\notag\\
&=\frac{1}{p}\tr[V_n^k]\ =\ \frac{1}{p}\tr\left[\left(\frac{1}{n}X_nX_n^T\right)^k\right]=\frac{1}{pn^k}\sum_{s\in\oneto{p}} (X_nX_n^T)^k(s,s)\notag\\
&= \frac{1}{pn^k}\sum_{s_1,\ldots,s_k\in\oneto{p}}(X_n X_n^T)(s_1,s_2)(X_nX_n^T)(s_2,s_3)\cdots(X_nX_n^T)(s_k,s_1)\notag\\
&=\frac{1}{pn^k}\sum_{s_1,\ldots,s_k\in\oneto{p}}\sum_{t_1,\ldots,t_k\in\oneto{n}} X_n(s_1,t_1)X_n(s_2,t_1)X_n(s_2,t_2)X_n(s_3,t_2)\cdots X_n(s_k,t_k)X_n(s_1,t_k)\notag\\
&=\frac{1}{pn^k} \sum_{\ubar{s}\in\oneto{p}^k}\sum_{\ubar{t}\in\oneto{n}^k} X_n(\ubar{s},\ubar{t}),\label{eq:elaboratesum}
\end{align}
where for all $\ubar{s}\in\oneto{p}^k$ and $\ubar{t}\in\oneto{n}^k$ we define
\begin{equation}
\label{eq:STwalk}
X_n(\ubar{s},\ubar{t}) \defeq X_n(s_1,t_1)X_n(s_2,t_1)X_n(s_2,t_2)X_n(s_3,t_2)\cdots X_n(s_k,t_k)X_n(s_1,t_k).
\end{equation}

As we saw in \eqref{eq:elaboratesum}, the random moments $\integrala{\mu_n}{x^k}$ expand into elaborate sums. In order to be able to analyze these sums, we sort them with the language of graph theory. Each pair $(\ubar{s},\ubar{t})\in\oneto{p}^k\times\oneto{n}^k$ spans a Eulerian bipartite graph as in Figure~\ref{fig:Eulerian}.
\begin{figure}[htbp]
\centering
\includegraphics[clip, trim=10cm 12cm 10cm 4cm, width=\textwidth]{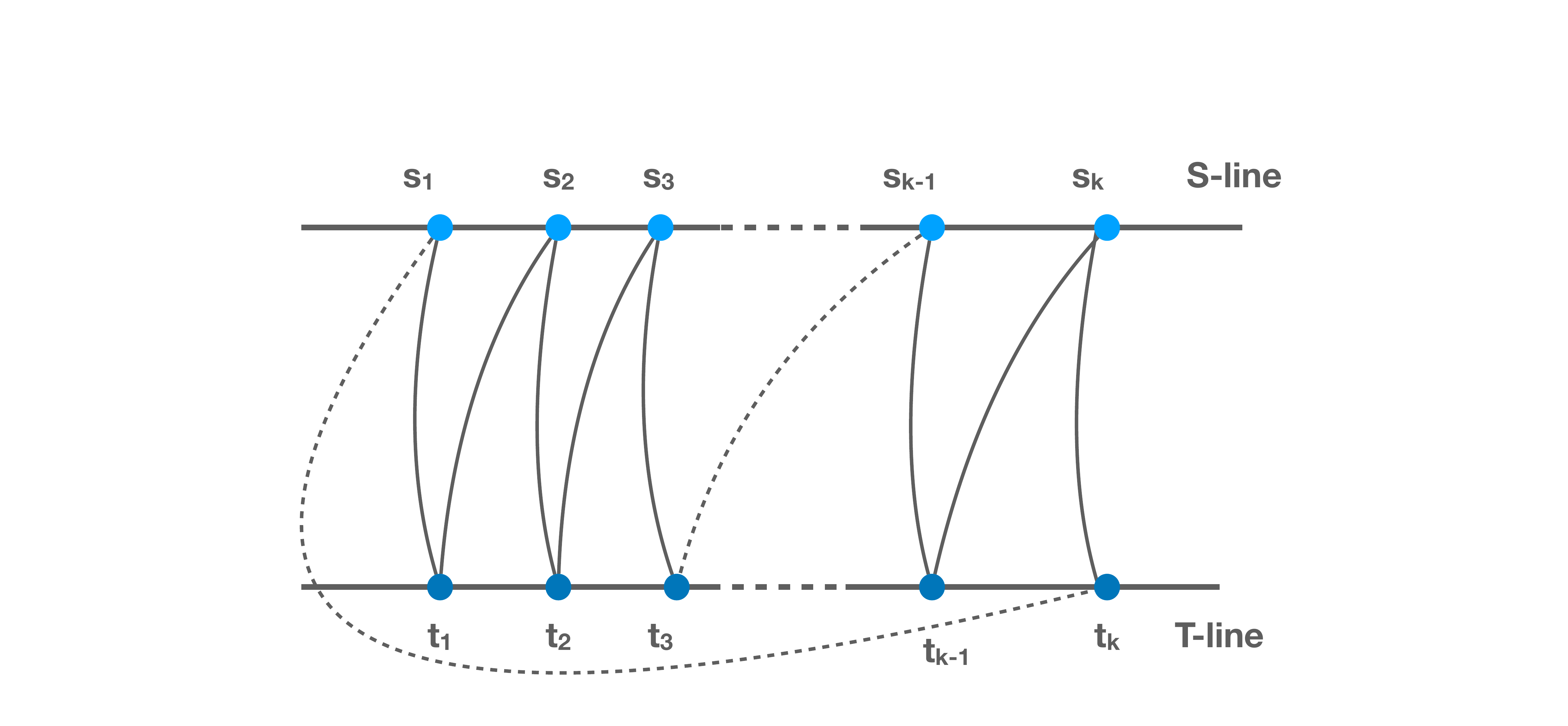}	
\caption{Eurlerian bipartite graph $\Gcal(\ubar{s},\ubar{t})$.}
\label{fig:Eulerian}
\end{figure}
Here, elements in the set $\{s_1,\ldots,s_k\}$ resp.\ $\{t_1,\ldots,t_k\}$ are called S-nodes resp.\ T-nodes. S- and T-nodes are considered different even if their value is the same and are thus placed on separate lines -- called S-line and T-line -- which are drawn horizontally beneath each other. Then we draw an undirected edge $\{s_i,t_j\}$ between $s_i$ and $t_j$, $i\in\oneto{p}$, $j\in\oneto{n}$, whenever $(s_i,t_j)$ or $(t_j,s_i)$ appears in \eqref{eq:STwalk}, where we allow for multi-edges. Here, if $(s_i,t_j)$ appears in \eqref{eq:STwalk}, we will call the edge $\{s_i,t_j\}$ \emph{down edge}, whereas if $(t_j,s_i)$ appears in \eqref{eq:STwalk}, we call $\{s_i,t_j\}$ \emph{up edge} (cf.\ Figure~\ref{fig:Eulerian}).
This yields the graph $\Gcal(\ubar{s},\ubar{t}) = (V(\ubar{s},\ubar{t}),E(\ubar{s},\ubar{t}),\phi_{\ubar{s},\ubar{t}})$, where
\begin{align*}
V(\ubar{s},\ubar{t})  &=\ \{s_1,\ldots,s_k\}\ \dot{\cup}\ \{t_1,\ldots,t_k\} \qquad \text{(disjoint union)}\\
E(\ubar{s},\ubar{t}) &= \{d_1,\ldots,d_k\}\ \dot{\cup}\ \{u_1,\ldots,u_k\}\qquad \text{(down edges, up edges)}	\\
&= \{e_1,e_2,\ldots, e_{2k}\} \qquad  (e_{2l-1} = d_l,\ e_{2l}=u_l,\ l=1,\ldots, k)\\
\phi_{\ubar{s},\ubar{t}}(d_i)&=\{s_i,t_i\}\\
\phi_{\ubar{s},\ubar{t}}(u_i)&=\{s_{i+1},t_i\}\,.
\end{align*}
Each $(\ubar{s},\ubar{t})$ also denotes a Eulerian cycle of length $2k$ through its graph $\Gcal(\ubar{s},\ubar{t})$ by 
\begin{equation}
\label{eq:Euleriancycle}	
s_1,d_1,t_1,u_1,s_2,d_2,t_2,\ldots,u_{k-1},s_k,d_k,t_k,u_k,s_1.
\end{equation}
Note that by construction, $\Gcal(\ubar{s},\ubar{t})$ contains no loops, but may contain multi-edges.
 The language of graph theory allows us to express $\integrala{\mu_n}{x^k}$ in a different way. Recall
\begin{equation}
\label{eq:naivemoment}	
\integrala{\mu_n}{x^k}\ =\ \frac{1}{pn^k} \sum_{\ubar{s}\in\oneto{p}^k}\sum_{\ubar{t}\in\oneto{n}^k} X_n(\ubar{s},\ubar{t})
\end{equation}
with
\begin{equation}
\label{eq:exmomentsummand}
 X_n(\ubar{s},\ubar{t}) = X_n(s_1,t_1)X_n(s_2,t_1)X_n(s_2,t_2)X_n(s_3,t_2)\cdots X_n(s_k,t_k)X_n(s_1,t_k).
\end{equation}
For any pair of tuples $(\ubar{s},\ubar{t})\in \oneto{p}^k\times \oneto{n}^k$, we define its profile
\[
\rho(\ubar{s},\ubar{t}) = (\rho_1(\ubar{s},\ubar{t}),\ldots,\rho_{2k}(\ubar{s},\ubar{t})),
\]
where for all $\ell\in[2k]$
\[
\rho_{\ell}(\ubar{s},\ubar{t}) = \#\{\phi_{\ubar{s},\ubar{t}}(e)\ |\ e\in E(\ubar{s},\ubar{t})\ \text{is an $\ell$-fold edge}\}.
\]
Here, an $\ell$-fold edge in $E(\ubar{s},\ubar{t})$ is any element $e\in E(\ubar{s},\ubar{t})$ for which there are exactly $\ell-1$ distinct other elements $e'_2,\ldots,e'_{\ell} \in E(\ubar{s},\ubar{t})$ so that $\phi_{\ubar{s},\ubar{t}}(e)=\phi_{\ubar{s},\ubar{t}}(e'_j)$ for $j\in\{2,\ldots,\ell\}$.

We now find that for all $\ell\in[2k]$, the Eulerian cycle $(\ubar{s},\ubar{t})$ traverses exactly $\rho_{\ell}(\ubar{s},\ubar{t})$ $\ell$-fold edges. As a result, the following trivial but useful equality holds:
\begin{equation}
\label{eq:alledgessum}
2k = \sum_{l=1}^{2k}\ell\cdot\rho_{\ell}(\ubar{s},\ubar{t}).
\end{equation}
Now for all $k\in\N$ we define the following set of profiles:
\[
\Pi(2k) = \left\{\rho \in\{0,\ldots,2k\}^{2k}\ |\ \rho \ \text{profile of some } (\ubar{s},\ubar{t})\in\oneto{p}^k\times\oneto{n}^k\right\}.
\]
Using this notation, we may write
\begin{equation}
\label{eq:graphmoment}	
\integrala{\mu_n}{x^k} =\sum_{\rho\in\Pi(2k)}\frac{1}{pn^k}\sum_{(\ubar{s},\ubar{t})\in\Tcal^{p,n}(\rho)}  X_n(\ubar{s},\ubar{t}),
\end{equation}
where
\[
\Tcal^{p,n}(\rho) \defeq \left\{(\ubar{s},\ubar{t})\in\oneto{p}^k\times\oneto{n}^k\ |\ \rho(\ubar{s},\ubar{t})=\rho \right\}.
\]
The transition from \eqref{eq:naivemoment} to \eqref{eq:graphmoment} allows us to keep track (in particular) of single and double edges since their contribution (or lack thereof) is a crucial point to analyze.

The next fundamental lemma will give an upper bound on the number of tuple pairs $(\ubar{s},\ubar{t})$ with at most $\ell\in\oneto{2k}$ vertices. Note that there are always at least two vertices present, since S-nodes and T-nodes are disjoint. Notationally, we set $V(\ubar{u})\defeq\{u_1,\ldots,u_k\}$ for any $\ubar{u}\in\N^k$ and if also $\ubar{v}\in\N^k$, we set $V(\ubar{u},\ubar{v})\defeq\{u_1,\ldots,u_k\}\,\dot{\cup} \,\{v_1,\ldots,v_k\}$, even if we do not regard $(\ubar{u},\ubar{v})$ as a graph. In particular, $\#V(\ubar{u},\ubar{v})= \#V(\ubar{u})+\#V(\ubar{v})$.

\begin{lemma}\label{lem:maxnodestuples}
Let $p,n,k\in\N$, $a,b\in\{1,\ldots,k\}$ and $\ell\in\{2,3,\ldots,2k\}$ be arbitrary. Then
\begin{align*}
i)&\quad 
\# \left\{(\ubar{s},\ubar{t})\in\oneto{p}^k\times\oneto{n}^k\, |\, \#V(\ubar{s})=a, \#V(\ubar{t})=b\right\} \,\leq\, (ab)^k\cdot p^a n^b \,\leq\, k^{2k}\cdot p^a n^b\\
ii)&\quad 
\# \left\{(\ubar{s},\ubar{t})\in\oneto{p}^k\times\oneto{n}^k\, |\, \#V(\ubar{s},\ubar{t})\leq \ell\right\} \,\leq\, \ell^{2k+2}\cdot (p\vee n)^{\ell} \,\leq\,  (2k)^{2k+2}\cdot (p\vee n)^{\ell}.
\end{align*}
\end{lemma}
\begin{proof}
For i) we check how many possibilities we have to construct such an ($\ubar{s},\ubar{t}$). First, we determine the coloring for $\ubar{s}$ by picking a surjective function $f:\{1,\ldots,k\}\to\{1,\ldots,a\}$, determining which places in $\ubar{s}$ should contain equal or different entries. This admits at most $a^k$ choices. Now we pick a value for $s_1$ and have $p$ possibilities. Then if $f(2)= f(1)$, we have no choice for $s_2$ since then $s_2$ must be equal to $s_1$. Otherwise, if $f(2)\neq f(1)$, we are left with at most $p$ choices for $s_2$. Proceeding this way for $\ell \in\{3,\ldots,k\}$, if $f(\ell)=f(i)$ for some $i<\ell$ then set $s_{\ell}\defeq s_i$, and otherwise we have at most $p-\#\{f(1),\ldots,f(\ell-1)\}$ choices for $s_{\ell}$, and these non-trivial choices happen $a-1$ times after the initial choice of $s_1$, thus admitting at most $p^a$ choices for $\ubar{s}$. Likewise, we pick a color structure $g:\{1,\ldots,k\}\to\{1,\ldots,b\}$ for $\ubar{t}$ for which we have at most $b^k$ choices and construct $\ubar{t}$ with at most $n^b$ choices. In total, we have at most
\[
 a^k\cdot p^a\cdot b^k\cdot n^b = (ab)^k\cdot p^a n^b
 \]
choices. This proves i), and for ii) we first decide on the number $a\leq k$ of different vertices in $\ubar{s}$ and the number $b\leq k$ of different vertices in $\ubar{t}$ such that $a+b\leq \ell$. This choice of $(a,b)$ admits at most $\ell^2$ choices. Then with i), the statement follows.
\end{proof}

\subsection{Convergence of expected moments}
In this subsection we establish \eqref{eq:expectation}, which will prove the MP law in expectation. To this end, we proceed to analyze
\begin{equation}
\label{eq:MPexpmom}
\E\integrala{\mu_n}{x^k} =\sum_{\rho\in\Pi(2k)}\frac{1}{pn^k}\sum_{(\ubar{s},\ubar{t})\in\Tcal^{p,n}(\rho)} \E X_n(\ubar{s},\ubar{t}).
\end{equation}
The way we analyze \eqref{eq:MPexpmom} is to establish upper bounds on $\E X_n(\ubar{s},\ubar{t})$ depending on its profile $\rho(\ubar{s},\ubar{t})$, and to bound $\#\Tcal^{p,n}(\rho)$ by a function of $p$, $n$ and $\rho$. For the latter, we formulate the next lemma. It is a modification of corresponding lemmas obtained in \cite{FleermannDiss} in the setting of random band matrices.

\begin{lemma}\label{lem:nodeandtuplecount}
Let $k \in\N$ be arbitrary. Then it holds $\#\Pi(2k)\leq 16^k$. Further, let $p,n\in\N$ and $\rho\in\Pi(2k)$ be arbitrary, then
\begin{enumerate}
\item[i)] For any $(\ubar{s},\ubar{t})\in\Tcal^{p,n}(\rho)$ we obtain
\[
\# V(\ubar{s},\ubar{t})\leq 1 + \rho_1 + \ldots + \rho_{2k} \,,
\]
so by Lemma~\ref{lem:maxnodestuples}:
\[
\#\Tcal^{p,n}(\rho)\leq (2k)^{2k+2} \cdot (p\vee n)^{1 + \rho_1 + \ldots + \rho_{2k}}.
\]
\item[ii)] If $\rho$ contains an odd edge, then for any $(\ubar{s},\ubar{t})\in\Tcal^{p,n}(\rho)$ we obtain
\[
\# V(\ubar{s},\ubar{t})\leq \rho_1 + \ldots + \rho_{2k}\,,
\]
so by Lemma~\ref{lem:maxnodestuples}:
\[
\#\Tcal^{p,n}(\rho)\leq (2k)^{2k+2} \cdot (p\vee n)^{\rho_1 + \ldots + \rho_{2k}}.
\]
\end{enumerate}
\end{lemma}

\begin{proof}
Each $\rho\in\Pi(2k)$ is a $2k$-tuple in which for all $\ell\in\{1,\ldots,2k\}$ the entry $\rho_{\ell}$ lies in the set $\{0,1,\ldots,\lfloor 2k/\ell\rfloor\}$, which follows directly from \eqref{eq:alledgessum}. Therefore, 
\[
\#\Pi(2k) \leq  \prod_{\ell=1}^{2k}\left(\frac{2k}{\ell}+1\right) = \frac{(4k)!}{(2k)!\cdot (2k)!} = \binom{2 (2k)}{2k}  \leq \frac{4^{2k}}{\sqrt{2k\pi}} \leq 16^k,
\]
where the fourth step is well-known fact about the central binomial coefficient. For the proof of statements $i)$ and $ii)$ it suffices to establish the upper bounds for $\#V(\ubar{s},\ubar{t})$, since the bounds on $\#\Tcal^{p,n}(\rho)$ then follow directly with Lemma~\ref{lem:maxnodestuples} $ii)$. Now to prove upper bounds for $\#V(\ubar{s},\ubar{t})$, the idea is to travel the Eulerian cycle generated by $(\ubar{s},\ubar{t}):$
\begin{equation}
\label{eq:walk}
s_1,e_1,t_1,e_2,s_2,e_3,t_2,\ldots,t_k, e_{2k}, s_1
\end{equation}
by picking an initial node $s_i$ or $t_i$ and then traversing the edges in increasing cyclic order until reaching the starting point again. On the way, we count the number of different nodes that were discovered. Whenever we pass an $\ell$-fold edge, only the first instance of that edge may discover a new vertex. \newline
\underline{$i)$}  We start our tour at  $s_1$ and observe this very vertex. Then, as we travel along the cycle, for each $\ell\in\{1,\ldots,2k\}$ we will pass $\ell\cdot \rho_{\ell}$ $\ell$-fold edges out of which only the first instance can discover a new node, and there are $\rho_{\ell}$ of these first instances. Considering the initial node, we arrive at $\# V(\ubar{s},\ubar{t})\leq 1 + \rho_1 + \ldots + \rho_{2k}$, which yields the desired inequality.\newline
\underline{$ii)$} In presence of an odd edge, we can start the tour at a specific vertex such that the odd edge cannot contribute to the newly discovered vertices. To this end, fix an $\ell$-fold edge in  $(\ubar{s},\ubar{t})$ with $\ell$ odd. Let $e_{i_1},\ldots,e_{i_{\ell}}$, $i_1 < \ldots < i_{\ell}$, be the instances of the $\ell$-fold edge in question in the cycle \eqref{eq:walk}. Since $\ell$ is odd, we must find a $j\in\{1,\ldots,\ell\}$ such that $e_{i_j}$ and $e_{i_{j+1}}$ are both up edges or both down edges (where $\ell+1\equiv 1$), since we are on a cycle. W.l.o.g.\ $e_{i_j}$ and $e_{i_{j+1}}$ are a down edges. We start our tour at the T-node after $e_{i_j}$. Since $e_{i_{j+1}}$ is a down-edge as well, its initial S-node must be discovered by an edge different from our $\ell$-fold edge.  In particular, none of the edges $e_{i_1},\ldots,e_{i_{\ell}}$ may discover a \emph{new} vertex. Therefore, the roundtrip leads to the discovery of at most $\rho_1 + \dots + (\rho_{\ell} - 1) + \ldots + \rho_{2k}$ new nodes in addition to the first node.
\end{proof}

Having established bounds on the quantities $\#\Tcal^{p,n}(\rho)$ in \eqref{eq:MPexpmom}, we now proceed to make the expressions $\E X_n(\ubar{s},\ubar{t})$ amenable for analysis. Note that by the generalized Hölder inequality, we can always apply the bound $\abs{\E X_n(\ubar{s},\ubar{t})}\leq L_{2k}\defeq (2k-1)!!$, which is the $2k$-th moment of a standard normal random variable. But we will have to bound $\E X_n(\ubar{s},\ubar{t})$ in more sophisticated ways. Recall that $X_n$ is a $p\times n$-matrix of jointly Gaussian entries with covariance matrix 
\[
\Sigma_{n}((a,b),(c,d)) \defeq \Cov(X_n(a,b),X_n(c,d)).
\]
for all $a,c\in\oneto{p}$, $b,d\in\oneto{n}$. In particular, $\Sigma_{n}$ is a $pn\times pn$ matrix, indexed by pairs.

To evaluate expectations of a product of correlated Gaussian random variables as in \eqref{eq:elaboratesum},  \eqref{eq:STwalk}, or \eqref{eq:MPexpmom}, we use the ``Theorem of Isserlis", also know as ``Wick's theorem", see e.g.~\cite{Speicher} or \cite{Isserlis}.
\begin{theorem}
\label{thm:wick}
Let $n\in\N$ and $\Sigma$ be a positive semidefinite, real symmetric $n\times n$ matrix. If $(Y_1,\ldots,Y_n)\sim \mathcal{N}(0,\Sigma)$, then for all $k\in\N$ and $i(1),\ldots,i(k)\in\oneto{n}$, it holds
\[
\E Y_{i(1)}\cdots Y_{i(k)} = \sum_{\pi\in\PePe{k}}\prod_{\{r,s\}\in\pi}\E Y_{i(r)}Y_{i(s)} = \sum_{\pi\in\PePe{k}}\prod_{\{r,s\}\in\pi} \Sigma \big(i(r),i(s)\big),
\]
where $\PePe{k}$\label{sym:pairpartitions} denotes the set of all pair partitions on $\{1,\ldots,k\}$. In particular, we obtain for $k$ odd that
\[
\E Y_{i(1)}\cdots Y_{i(k)} = 0.
\]
\end{theorem}

Using the formula of Isserlis, we can give a more detailed version of \eqref{eq:MPexpmom} as follows:
\begin{align}
\E\integrala{\mu_n}{x^k} 
&=\sum_{\rho\in\Pi(2k)}\frac{1}{pn^k}\sum_{(\ubar{s},\ubar{t})\in\Tcal^{p,n}(\rho)} \E X_n(\ubar{s},\ubar{t})\notag\\
&= \sum_{\rho\in\Pi(2k)}\frac{1}{pn^k}\sum_{(\ubar{s},\ubar{t})\in\Tcal^{p,n}(\rho)}\sum_{\pi\in\Pcal\Pcal(2k)}\prod_{\{u,v\}\in\pi} \Sigma_n\left(e_u(\ubar{s},\ubar{t}),e_v(\ubar{s},\ubar{t})\right),\label{eq:MPexpGaussianmom}
\end{align}
where for any $w\in\oneto{2k}$, $e_w(\ubar{s},\ubar{t})$  is the pair $(a,b)\in \oneto{p}\times\oneto{n}$ containing the nodes  connected by the $w$-th edge of $(\ubar{s},\ubar{t})$.  
Next, let us analyze for which $\rho\in\Pi(2k)$ we can expect an asymptotic contribution in \eqref{eq:MPexpGaussianmom}.\newline
\underline{Case 1: $\rho_1=0$ and $\rho_{\ell}>0$ for some $\ell\geq 3$.}\newline
We obtain
\[
1 + \rho_1 + \ldots + \rho_{2k} \leq  
\left\{
\begin{array}{c}
 1 + \frac{2k-6}{2} + 2\\
 1 + \frac{2k-4}{2} + 1
\end{array}
\right\} = k,
\]
where the upper case is valid in presence of an odd edge (in which we then find at least a second odd edge), and the lower case is valid if no odd edges are present. 
Therefore, $\#\Tcal^{p,n}(\rho) \leq (2k)^{2k+2}\cdot (p\vee n)^k$ by Lemma~\ref{lem:nodeandtuplecount}, and since each summand $\E X_n(\ubar{s},\ubar{t})$ is uniformly bounded by some constant $L_{2k}$, we find a contribution of order $O_k(n^{-1})$, where $O_k(\cdot)$ means $O(\cdot)$ with a constant which depends only on $k$.\newline
\underline{Case 2: $\rho_2=k$.} In this case, $\rho$ admits only double edges, and we denote this specific profile by $\rho^{(k)}$, so $\rho^{(k)}_2=k$ and $\rho^{(k)}_{\ell}=0$ for all $\ell\neq 2$. Then by Lemma~\ref{lem:nodeandtuplecount}, all $(\ubar{s},\ubar{t})\in\Tcal^{p,n}(\rho^{(k)})$ have at most $k+1$ nodes, so we may subdivide this set further. We define
\begin{align*}
\Tcal^{p,n}_{\leq k}(\rho^{(k)}) &\defeq \left\{(\ubar{s},\ubar{t})\in\Tcal^{p,n}(\rho^{(k)}):\ \# V(\ubar{s},\ubar{t}) \leq k \right\}\ ,\\
\Tcal^{p,n}_{k+1}(\rho^{(k)}) &\defeq \left\{(\ubar{s},\ubar{t})\in\Tcal^{p,n}(\rho^{(k)}):\ \# V(\ubar{s},\ubar{t}) = k+1 \right\}
\end{align*}
and note that by Lemma~\ref{lem:maxnodestuples}, $\#\Tcal^{p,n}_{\leq k}(\rho^{(k)}) \leq (2k)^{2k+2}(p\vee n)^{k}$, so that 
\[
\frac{1}{pn^k}\sum_{(\ubar{s},\ubar{t})\in\Tcal^{p,n}(\rho^{(k)})} \E X_n(\ubar{s},\ubar{t}) =\frac{1}{pn^k} \sum_{(\ubar{s},\ubar{t})\in\Tcal_{k+1}^{p,n}(\rho^{(k)})} \E X_n(\ubar{s},\ubar{t}) \ + \ O_k\left(\frac{1}{n}\right).
\]
Now we use the formula of Isserlis again to obtain
\begin{equation}
\label{eq:GaussianMPmomentsOne}	
 \frac{1}{pn^k}\sum_{(\ubar{s},\ubar{t})\in\Tcal_{k+1}^{p,n}(\rho^{(k)})} \E X_n(\ubar{s},\ubar{t}) = \frac{1}{pn^k} \sum_{(\ubar{s},\ubar{t})\in\Tcal_{k+1}^{p,n}(\rho^{(k)})} \sum_{\pi\in\Pcal\Pcal(2k)}\prod_{\{u,v\}\in\pi} \Sigma_n\left(e_u(\ubar{s},\ubar{t}),e_v(\ubar{s},\ubar{t})\right).
\end{equation}
We observe in \eqref{eq:GaussianMPmomentsOne} that for each $(\ubar{s},\ubar{t})\in\Tcal_{k+1}^{p,n}(\rho^{(k)})$, there is exactly one pair partition $\pi^*=\pi^*(\ubar{s},\ubar{t})\in\Pcal\Pcal(2k)$ which pairs all the double edges in $(\ubar{s},\ubar{t})$, the subsequent product then assuming the value $1$. For all finitely many other $\pi\neq \pi^*$, we find at least two blocks in $\pi$ which pair two different edges, thus leading to at least two factors of off-diagonal entries in ${\Sigma_n}$, thus at least to a decay of order $O_k(a_n^2/n^2)$. So for all $(\ubar{s},\ubar{t})\in\Tcal_{k+1}^{p,n}(\rho^{(k)})$,
\[
\bigabs{\sum_{\pi\in\Pcal\Pcal(2k)}\prod_{\{u,v\}\in\pi} \Sigma_n\left(e_u(\ubar{s},\ubar{t}),e_v(\ubar{s},\ubar{t})\right) \ - \  1} \leq \#\Pcal\Pcal(2k) \cdot\frac{a_n^2}{n^2}.
\] 
Therefore, \eqref{eq:GaussianMPmomentsOne} becomes
\begin{equation}
\label{eq:GaussianMPmomentsTwo}	
\frac{1}{pn^k}\sum_{(\ubar{s},\ubar{t})\in\Tcal_{k+1}^{p,n}(\rho^{(k)})} \E X_n(\ubar{s},\ubar{t})  =  \frac{1}{pn^k}\# \Tcal_{k+1}^{p,n}(\rho^{(k)}) + \frac{1}{pn^k}O_{\star}\left(\#\Tcal_{k+1}^{p,n}(\rho^{(k)})\cdot \#\Pcal\Pcal(2k)\cdot \frac{a_n^2}{n^2} \right),
\end{equation}
where $O_{\star}(\cdot)$ means $O(\cdot)$ with constant $1$.

Since by Lemma~\ref{lem:nodeandtuplecount},
\[
\#\Tcal_{k+1}^{p,n}(\rho^{(k)}) \leq (2k)^{2k+2}(p\vee n)^{1+k},
\]
we obtain
\begin{equation}
\label{eq:GaussianMPmomentsThree}	
\frac{1}{pn^k}\sum_{(\ubar{s},\ubar{t})\in\Tcal_{k+1}^{p,n}(\rho^{(k)})} \E X_n(\ubar{s},\ubar{t})  = \frac{1}{pn^k}\cdot \# \Tcal_{k+1}^{p,n}(\rho^{(k)}) +   O_k\left(\frac{a_n^2}{n^2}\right).
\end{equation}
It is well-known that (e.g.\ \cite{Fleermann:Kirsch:2022c} or from the analysis in \cite{BaiSi})
\[
\frac{1}{pn^k}\cdot \# \Tcal_{k+1}^{p,n}(\rho^{(k)}) = \frac{1}{pn^k}\sum_{r=0}^{k-1}\frac{(p)_{r+1}(n)_{k-r}}{r+1}\binom{k}{r}\binom{k-1}{r}\ \xrightarrow[n\to\infty]{} \sum_{r=0}^{k-1} \frac{y^r}{r+1}\binom{k}{r}\binom{k-1}{r},
\]
which is the $k$-th moment of the MP distribution $\mu^y$, cf.\ \eqref{eq:MPmoments}. Here, for any integers $0\leq k\leq \ell$, $(\ell)_k \defeq \ell\cdot(\ell-1)\cdots (\ell-k+1)$, where an empty product equals $1$.

We have now completely established the contribution in \eqref{eq:MPexpGaussianmom} stemming from those $\rho\in\Pi(2k)$ with $\rho_1=0$. Their contribution is either vanishing or -- in the case $\rho_2=k$ -- yielding the moments of the Marchenko-Pastur distribution. Therefore, we will have shown \eqref{eq:expectation} if we can show that those $\rho\in\Pi(2k)$ with $\rho_1>0$ have a vanishing contribution. \newline
\underline{Case 3: $\rho_1>0$.}\newline
Since $\rho_1>0$, we obtain by Lemma~\ref{lem:nodeandtuplecount} that
\[
\#\Tcal^{p,n}(\rho)\leq (2k)^{2k+2} \cdot (p\vee n)^{\rho_1 + \ldots + \rho_{2k}} \leq (2k)^{2k+2} (n\vee p)^{\frac{\rho_1}{2} + k}.
\]
On the other hand, for every $(\ubar{s},\ubar{t})\in\Tcal^{p,n}(\rho)$, each pair partition $\pi\in\Pcal\Pcal(2k)$ has at least $\lceil\rho_1/2\rceil$ blocks leading to off-diagonal entries of $\Sigma_n$, thus yielding a decay of $(a_n/n)^{\rho_1/2}$ or faster in the product in \eqref{eq:MPexpGaussianmom}. As a result,
\begin{multline*}
\frac{1}{pn^k}\sum_{(\ubar{s},\ubar{t})\in\Tcal^{p,n}(\rho)}\sum_{\pi\in\Pcal\Pcal(2k)}\prod_{\{u,v\}\in\pi} \Sigma_n\left(e_u(\ubar{s},\ubar{t}),e_v(\ubar{s},\ubar{t})\right)\\
 \leq \frac{(2k)^{2k+2}(n\vee p)^{\frac{\rho_1}{2} + k}}{pn^k}\cdot \#\Pcal\Pcal(2k)\cdot\left(\frac{a_n}{n}\right)^{\frac{\rho_1}{2}} = O_k\left(\frac{a_n^{\frac{\rho_1}{2}}}{n}\right).
\end{multline*}
Therefore, we see that under the condition that
$\abs{\Sigma_n\left((a,b),(c,d)\right)} \leq a_n/n$
for all $(a,b)\neq(c,d)$, all $\rho$ with $\rho_1\geq 1$ will not contribute to the expected moments asymptotically. This proves the MP law in expectation, that is, \eqref{eq:expectation}, under assumptions (A1), (A2) and (A3).

\subsection{Divergence of expected moments}
In this subsection we prove the second statement of Theorem~\ref{thm:MPlaw}. We will show that if in the condition
\begin{equation}
\label{eq:counterexample}	
\forall\, n\in\N:\, \forall\,(a,b)\neq(c,d)\in\oneto{p}\times\oneto{n}:\quad \bigabs{\Sigma_n\left((a,b),(c,d)\right)} \leq \frac{a_n}{n},
\end{equation}
we do not require that $a_n=o(n^{\epsilon})$ for all $\epsilon>0$, that then there is an ensemble $Y_n$ of correlated Gaussian data matrices the entries of which have covariance matrix $\Sigma_n$ satisfying this condition, but where all but finitely many moments diverge to infinity.
To this end, let $a_n\in\R_+$ and $\epsilon>0$ such that $a_n\neq o(n^{\epsilon})$. Then there is a $c>0$ and a subsequence $J\subseteq\N$ such that $a_n\geq c n^{\epsilon}$ for all $n\in J$. W.l.o.g.\ we may assume that $\epsilon <1$ and $c < 1$. Then
define
\[
\forall\, n\in\N:\, \forall\,(a,b), (c,d)\in\oneto{p}\times\oneto{n}:\quad \Sigma_n\left((a,b),(c,d)\right) =
\begin{cases}
1 &\text{if } (a,b)=(c,d),\\
\frac{cn^{\epsilon}}{n} &\text{if } (a,b)\neq(c,d) \text{ and } n\in J,\\
0 &\text{if } (a,b)\neq(c,d)\text{ and }n\notin J.
\end{cases}
\] 
Then surely, the sequence $\Sigma_n$ is a sequence of positive definite covariance matrices, and it satisfies \eqref{eq:counterexample}.
Inspecting the sum \eqref{eq:MPexpGaussianmom}, we obtain with the same arguments as above, that the profiles $\rho\in\Pi(2k)$ with $\rho_1=0$ will yield the MP moments asymptotically (in the case $\rho_2=k$) or have a vanishing contribution (in the case $\rho_{\ell}>0$ for some $\ell\geq 3$). We now observe that all summands in \eqref{eq:MPexpGaussianmom} are positive, so it suffices to identify a profile $\rho\in\Pi(2k)$ with $\rho_1\geq 1$ and a contribution that diverges to infinity as $n\to\infty$ for all $k\in\N$ large enough. To this end, let $\rho^*\in\Pi(2k)$ be the profile with $\rho^*_1=2k$ and $\Tcal^{p,n}_{*}(\rho^*)\subseteq \Tcal^{p,n}(\rho^*)$ be the subset of all $(\ubar{s},\ubar{t})\in\Tcal^{p,n}(\rho^*)$ with $\#V(\ubar{s},\ubar{t})=2k$. Then 
\[
\#\Tcal^{p,n}_{*}(\rho^*) = (p)_k\cdot (n)_k \sim p^k n^k.
\]
On the other hand, for all $n\in J$, $(\ubar{s},\ubar{t})\in\Tcal^{p,n}_{*}(\rho^*)$, and $\pi\in\Pcal\Pcal(2k)$, 
\[
\prod_{\{u,v\}\in\pi} \Sigma_n\left(e_u(\ubar{s},\ubar{t}),e_v(\ubar{s},\ubar{t})\right) = \frac{c^k n^{k\epsilon}}{n^k}.
\]
Therefore, for $n\in J$,
\[
\frac{1}{pn^k}\sum_{(\ubar{s},\ubar{t})\in\Tcal^{p,n}_{*}(\rho)}\sum_{\pi\in\Pcal\Pcal(2k)}\prod_{\{u,v\}\in\pi} \Sigma_n\left(e_u(\ubar{s},\ubar{t}),e_v(\ubar{s},\ubar{t})\right) = \frac{(p)_k(n)_k}{pn^k}\cdot\frac{c^kn^{k\epsilon}\#\Pcal\Pcal(2k)}{n^k}\longrightarrow\infty 
\]
as soon as $k$ is chosen so large that $k\epsilon>1$. More precisely, for all $k>1/\epsilon$, $\E\integrala{\mu_n}{x^k}\to\infty$ for $n\in J$.

\subsection{Decay of the variances}
When analyzing the expectation of the random moments, we used the finite decomposition of the random moment as in \eqref{eq:graphmoment} and showed that for each $\rho\in\Pi(2k)$, the \emph{$\rho$-part} of the random moment
\begin{equation}
\label{eq:rhopart}
\frac{1}{pn^k}\sum_{(\ubar{s},\ubar{t})\in\Tcal^{p,n}(\rho)}  X_n(\ubar{s},\ubar{t})	
\end{equation}
converges in expectation to a constant; it converges to the $k$-th moment of the Marchenko-Pastur distribution if $\rho=\rho^{(k)}$, and to zero if $\rho\neq\rho^{(k)}$. To establish that moments converge almost surely, it thus suffices to show that the variance of each $\rho$-part \eqref{eq:rhopart} converges summably fast to zero. This variance is given by
\begin{equation}
\label{eq:varrhopart}
\frac{1}{p^2n^{2k}}\sum_{(\ubar{s},\ubar{t}),(\ubar{s}',\ubar{t}')\in\Tcal^{p,n}(\rho)} \left[\E X_n(\ubar{s},\ubar{t})X_n(\ubar{s}',\ubar{t}')- \E X_n(\ubar{s},\ubar{t})\E X_n(\ubar{s}',\ubar{t}')\right]\,.
\end{equation}
Now by Isserlis' formula,
\begin{equation}
\label{eq:exprod}	
\E X_n(\ubar{s},\ubar{t})X_n(\ubar{s}',\ubar{t}') = \sum_{\vartheta\in\Pcal\Pcal(4k)}\prod_{\{a,b\}\in\vartheta}\Sigma_n(e_a(\ubar{s},\ubar{t},\ubar{s}',\ubar{t}'),e_b(\ubar{s},\ubar{t},\ubar{s}',\ubar{t}')),
\end{equation}
where for any $w\in\oneto{4k}$, $e_w(\ubar{s},\ubar{t},\ubar{s}',\ubar{t}')$  is the pair $(a,b)\in \oneto{p}\times\oneto{n}$ containing the nodes connected by the $w$-th edge of $(\ubar{s},\ubar{t})$ if $w\in\oneto{2k}$ or by the $(w-2k)$-th edge of $(\ubar{s}',\ubar{t}')$ if $w\in\oneto{4k}\backslash\oneto{2k}$. Further, we obtain
\begin{equation}
\label{eq:prodex}
\E X_n(\ubar{s},\ubar{t})\E X_n(\ubar{s}',\ubar{t}') = 	\sum_{\substack{\pi\in\Pcal\Pcal(2k)\\ \pi'\in\Pcal\Pcal(2k)}} \prod_{\substack{\{u,v\}\in\pi\\ \{u',v'\}\in\pi'}}\Sigma_n(e_u(\ubar{s},\ubar{t}),e_v(\ubar{s},\ubar{t}))\Sigma_n(e_{u'}(\ubar{s}',\ubar{t}'),e_{v'}(\ubar{s}',\ubar{t}')).
\end{equation}
We will call a partition $\vartheta\in\Pcal\Pcal(4k)$ \emph{contained}, if any block $B\in\vartheta$ is either a subset of $\oneto{2k}$ or of $\oneto{4k}\backslash\oneto{2k}$. 
The key idea now is to associate any combination $(\pi,\pi')\in\Pcal\Pcal(2k)\times \Pcal\Pcal(2k)$ to the corresponding unique contained partition $\vartheta \in \Pcal\Pcal(4k)$ by defining the function $f:\oneto{2k}\to \oneto{4k}\backslash\oneto{2k}$, $f(\ell) = 2k + \ell$, and then 
\[
\vartheta \defeq \pi\cup \left\{ f(B)\ | \ B\in\pi'\right\}.
\]
Denote by $\Ncal\Pcal\Pcal(4k)$ the set of \emph{non-contained} pair partitions of $\oneto{4k}$. That is, each $\vartheta\in\Ncal\Pcal\Pcal(4k)$ has at least one block (and thus at least two blocks) $B=\{i,j\}$ where $i\in\oneto{2k}$ and $j\in\oneto{4k}\backslash\oneto{2k}$. Blocks with this property will be called \emph{traversing} in what follows. Note that each $\vartheta\in\Ncal\Pcal\Pcal(4k)$ has at least two traversing blocks.
We observe
\begin{multline}
\label{eq:difference}
\E X_n(\ubar{s},\ubar{t})X_n(\ubar{s}',\ubar{t}') - \E X_n(\ubar{s},\ubar{t})\E X_n(\ubar{s}',\ubar{t}') \\
= \sum_{\vartheta\in\Ncal\Pcal\Pcal(4k)}\prod_{\{a,b\}\in\vartheta}\Sigma_n(e_a(\ubar{s},\ubar{t},\ubar{s}',\ubar{t}'),e_b(\ubar{s},\ubar{t},\ubar{s}',\ubar{t}')).
\end{multline}
Thus, \eqref{eq:varrhopart} becomes
\begin{equation}
\label{eq:varanalysis}
\frac{1}{p^2n^{2k}}\sum_{(\ubar{s},\ubar{t}),(\ubar{s}',\ubar{t}')\in\Tcal^{p,n}(\rho)} \sum_{\vartheta\in\Ncal\Pcal\Pcal(4k)}\prod_{\{a,b\}\in\vartheta}\Sigma_n(e_a(\ubar{s},\ubar{t},\ubar{s}',\ubar{t}'),e_b(\ubar{s},\ubar{t},\ubar{s}',\ubar{t}')).		
\end{equation}
It is our goal to show that \eqref{eq:varanalysis} converges to zero summably fast. To this end, we define
\[
\Tcal_{d}^{p,n}(\rho)\defeq \left\{((\ubar{s},\ubar{t}),(\ubar{s}',\ubar{t}')) \in (\Tcal^{p,n}(\rho))^2\,|\,\text{$(\ubar{s},\ubar{t})$ and $(\ubar{s}',\ubar{t}')$ are edge-disjoint}\right\}
\]
and for all $\ell\in\oneto{2k}$,
\begin{multline*}
\Tcal_{c(\ell)}^{p,n}(\rho)\defeq \left\{((\ubar{s},\ubar{t}),(\ubar{s}',\ubar{t}'))\in(\Tcal^{p,n}(\rho))^2\,|\right.\\ \left.\text{$(\ubar{s},\ubar{t})$ and $(\ubar{s}',\ubar{t}')$ have exactly $\ell$ edges in common}\right\}.
\end{multline*}
In \eqref{eq:varanalysis} we will now consider the partial sums over $\Tcal_{d}^{p,n}(\rho,\rho')$ and then over $\Tcal_{c(\ell)}^{p,n}(\rho,\rho')$ for all $\ell = 1,\ldots, 2k$.\newline
\underline{Case 1: The edge-disjoint case.}\newline
\underline{Subcase 1: $\rho$ has only even edges.}\newline
Then by Lemma~\ref{lem:nodeandtuplecount},
\[
\#\Tcal_{d}^{p,n}(\rho) \leq \#\Tcal^{p,n}(\rho) \cdot \#\Tcal^{p,n}(\rho) \leq (2k)^{4k+4}(p\vee n)^{2k+2}.
\]
Since every non-contained $\vartheta$ will lead to at least 2 factors of decay in the product in \eqref{eq:varanalysis}, we obtain that the partial sum in \eqref{eq:varanalysis} over $\Tcal_{d}^{p,n}(\rho)$ is bounded by: 
\[
\frac{1}{p^2n^{2k}}\cdot (2k)^{4k+4}(p\vee n)^{2k+2} \cdot \#\Ncal\Pcal\Pcal(4k)\cdot \frac{a_n^2}{n^2} \xrightarrow[n\to\infty]{} 0 \quad \text{summably fast.}
\]
\underline{Subcase 2: $\rho$ has an odd edge.}\newline
Then by Lemma~\ref{lem:nodeandtuplecount},
\[
\#\Tcal_{d}^{p,n}(\rho) \leq \#\Tcal^{p,n}(\rho) \cdot \#\Tcal^{p,n}(\rho) \leq (2k)^{4k+4}(p\vee n)^{2k + \rho_1}.
\]
Also, we obtain at least $\rho_1\vee 2$ factors of decay in the product in \eqref{eq:varanalysis}, since in the worst case, $\vartheta$ groups two single edges of $(\ubar{s},\ubar{t})$ each with a single edge of $(\ubar{s}',\ubar{t}')$ and all remaining single edges into pairs. Thus,  the partial sum in \eqref{eq:varanalysis} over $\Tcal_{d}^{p,n}(\rho)$ is bounded by:
\[
\frac{1}{p^2n^{2k}}\cdot (2k)^{4k+4}(p\vee n)^{2k + \rho_1} \cdot\#\Ncal\Pcal\Pcal(4k)\cdot \frac{a_n^2}{n^{\rho_1\vee 2}} \xrightarrow[n\to\infty]{} 0 \quad \text{summably fast.}
\]
\underline{Case 2: The common-edge case.}\newline
In the edge-disjoint case, the fact that we considered only contained partitions $\vartheta\in\Ncal\Pcal\Pcal(4k)$ helped in that the (at least) two traversing blocks led to two factors of decay in the product of covariances in \eqref{eq:varanalysis}. In the common-edge case, this advantage is diminished since it is possible the traversing blocks in $\vartheta$ group common (i.e.\ same) edges in $(\ubar{s},\ubar{t})$ and $(\ubar{s}',\ubar{t}')$, leading to factors of $1$ in the product of covariances. To make up for this diminished decay, we need good bounds on the quantities $\#\Tcal_{c(\ell)}^{p,n}(\rho,\rho')$, which is the content of the following lemma:

\begin{lemma}\label{lem:doublebound}
Let $\rho\in\Pi(2k)$ and $\ell\in\oneto{2k}$, then the following statements hold:
\begin{enumerate}[i)]
\item For all $(\ubar{s},\ubar{t}),(\ubar{s}',\ubar{t}')\in\Tcal^{p,n}(\rho)$ with at least $\ell$ common edges, it holds
\[
\#(V(\ubar{s},\ubar{t})\cup V(\ubar{s}',\ubar{t}')) \leq 1 + 2\sum_{i=1}^{2k} \rho_i - \ell,
\]
so by Lemma~\ref{lem:maxnodestuples}:
\[
\#\Tcal_{c(\ell)}^{p,n}(\rho) \leq (2k)^{4k+2} (n\vee p)^{1 + 2\sum_{i=1}^{2k} \rho_i - \ell}.
\]
\item If there is an $\ell\in\oneto{2k}$ odd with $\rho_{\ell}\geq 1$, then for all $(\ubar{s},\ubar{t}),(\ubar{s}',\ubar{t}')\in\Tcal^{p,n}(\rho)$ with at least $\ell$ common edges, it holds
\[
\#(V(\ubar{s},\ubar{t})\cup V(\ubar{s}',\ubar{t}')) \leq 2\sum_{i=1}^{2k} \rho_i  - \ell,
\]
so by Lemma~\ref{lem:maxnodestuples}:
\[
\#\Tcal_{c(\ell)}^{p,n}(\rho) \leq (2k)^{4k+2} (n\vee p)^{2\sum_{i=1}^{2k} \rho_i - \ell}.
\]
\end{enumerate}
\end{lemma}
\begin{proof}
For statement $ii)$ we assume w.l.o.g.\ that $(\ubar{s},\ubar{t})$ has an odd edge. Since the graphs spanned by $(\ubar{s},\ubar{t})$ and $(\ubar{s}',\ubar{t}')$ share $\ell\geq 1$ common edges, we may take a tour around the joint Eulerian circle, starting before a common edge, traveling first all edges of $(\ubar{s},\ubar{t})$ and then all edges of $(\ubar{s}',\ubar{t}')$. While walking the edges of $(\ubar{s},\ubar{t})$, we can see at most $\rho_1 + \ldots + \rho_{2k}$ different nodes by Lemma~\ref{lem:nodeandtuplecount}. Next, traveling all edges of $(\ubar{s}',\ubar{t}')$, at most all the single edges and first instances of $m$-fold edges with $m\in\{2,\ldots,2k\}$ of $(\ubar{s}',\ubar{t}')$ may discover a new node, but only if they have not been traversed before during the walk along $(\ubar{s},\ubar{t})$. Since we have $\ell$ common edges, we can see at most $\rho'_1+\ldots + \rho'_{2k}-\ell$ new nodes. We established the bounds on the number of vertices in $ii)$. The second statement in $ii)$ follows immediately with Lemma~\ref{lem:maxnodestuples} $ii)$ by concatenating $(\ubar{s},\ubar{s}')\in\oneto{p}^{2k}$ and $(\ubar{t},\ubar{t}')\in\oneto{n}^{2k}$. For statement $i)$ we proceed exactly in the same manner: Traveling $(\ubar{s},\ubar{t})$ we can see at most $1 + \rho_1 + \rho_2 + \ldots + \rho_{2k}$ nodes by Lemma~\ref{lem:nodeandtuplecount}, then traveling $(\ubar{s}',\ubar{t}')$ we can see at most $\rho'_1+\ldots + \rho'_{2k}-\ell$ new nodes. Now apply Lemma~\ref{lem:maxnodestuples} $ii)$ again.
\end{proof}

In the following subcases, we assume that the number of common edges $\ell\in\{1,\ldots,2k\}$ is fixed and that we consider the subsum over $\Tcal_{c(\ell)}^{p,n}(\rho)$ in \eqref{eq:varanalysis}.\newline
\underline{Subcase 1: $\rho$ has only even edges.}\newline
Then by Lemma~\ref{lem:doublebound},
\[
\#\Tcal_{c(\ell)}^{p,n}(\rho) \leq (2k)^{4k+2} (n\vee p)^{1 + 2k - \ell}\leq (2k)^{4k+2} (n\vee p)^{2k}.
\]
Therefore, the partial sum in \eqref{eq:varanalysis} over $\Tcal_{c(\ell)}^{p,n}(\rho)$ is bounded by
\[
\frac{1}{p^2n^{2k}}\cdot (2k)^{4k+2} (n\vee p)^{2k} \cdot \#\Ncal\Pcal\Pcal(4k) \xrightarrow[n\to\infty]{} 0 \quad \text{summably fast.}
\]
\underline{Subcase 2: $\rho$ has an odd edge.}
\newline
Then by Lemma~\ref{lem:doublebound},
\[
\#\Tcal_{c(\ell)}^{p,n}(\rho) \leq (2k)^{4k+2} (n\vee p)^{\rho_1 + 2k-\ell}.
\]
Now let $\vartheta\in\Ncal\Pcal\Pcal(4k)$ be fixed. Then for each common edge, $\vartheta$ might pair a single edge in $(\ubar{s},\ubar{t})$ and the same single edge in $(\ubar{s}',\ubar{t}')$, leading to a factor of $1$ in the product of covariances. This can happen at most $\ell$ times in the worst case. Also in the worst case, $\vartheta$ could then pair all the remaining $(\rho_1-\ell)\vee 0$ single edges in $(\ubar{s},\ubar{t})$ with each other, and likewise the remaining $(\rho_1-\ell)\vee 0$ single edges in $(\ubar{s}',\ubar{t}')$ with each other, leading to at least $[(\rho_1-\ell)\vee 0 + (\rho_1-\ell)\vee 0]/2= (\rho_1-\ell)\vee 0 $ factors of decay in the product of covariances in \eqref{eq:varanalysis}. 
Therefore, the partial sum in \eqref{eq:varanalysis} over $\Tcal_{c(\ell)}^{p,n}(\rho)$ is bounded by
\[
 \frac{1}{p^2n^{2k}}\cdot (2k)^{4k+2} (n\vee p)^{2k+\rho_1 -\ell} \cdot \#\Ncal\Pcal\Pcal(4k)\cdot \frac{a_n^{(\rho_1-\ell)\vee 0}}{n^{(\rho_1-\ell)\vee 0}}  \xrightarrow[n\to\infty]{} 0 \quad \text{summably fast.}
\]
This concludes the proof.

\section{Proof of Theorems~\ref{thm:opnorm} and~\ref{thm:rankoneexample}: Analysis of the Largest Eigenvalue.}
\label{sec:Eval}

In this section, we analyze the largest eigenvalue of $V_n=\frac{1}{n}X_n X_n^T$ where $X_n$ satisfies (A1), \eqref{eq:deltacondition} and (A3), so that the correlations of the ensemble are uniformly bounded as follows:
\begin{equation*}
\forall\ (a,b)\neq(c,d)\in\oneto{p}\times\oneto{n}:\ \abs{\E X_n(a,b)X_n(c,d)} = \abs{\Sigma_n((a,b),(c,d))} \leq \frac{C}{n^{\delta}}.\label{eq:corrdecay2}
\end{equation*}
for some fixed $C,\delta>0$, where we assume for notational simplicity and w.l.o.g.\ that $C=1$. We will show that if $\delta>1$ we can guarantee almost sure convergence to the operator norm of $V_n$ to $y_+\defeq (1+\sqrt{y})^2$, which is the maximum of the bounded support of the Marchenko-Pastur distribution $\mu^y$, where $y=\lim_{n\to \infty} p/n$. On the other hand, if $\delta<1$, we can find an equicovariant ensemble matching above mentioned conditions and for which the operator norm of $V_n$ diverges to infinity almost surely. Lastly, if $\delta=1$, we can find an ensemble matching (A1), \eqref{eq:deltacondition}, and (A3), for which the operator norm remains stochastic in the limit. The precise results for the cases $\delta< 1$ and $\delta=1$ constitute Theorem~\ref{thm:rankoneexample}, while qualitative versions of these results are mentioned in parts $ii)$ and $iii)$ of Theorem~\ref{thm:opnorm}. This section is organized as follows: In Subsection~\ref{sec:convergence}, we show convergence of the operator norm for $\delta>1$, thus proving Theorem~\ref{thm:opnorm} $i)$. A key lemma which is used here is proved in Subsection~\ref{sec:boundproof}.
Subsection~\ref{sec:divergence} contains the proof of Theorem~\ref{thm:rankoneexample}, where we construct the counterexamples for the cases $\delta<1$ and $\delta=1$. From these counterexamples, statements $ii)$ and $iii)$ of Theorem~\ref{thm:opnorm} follow.

\subsection{Convergence of the operator norm}
\label{sec:convergence}

To show convergence of the operator norm in the case $\delta>1$, we follow the general strategy outlined in \cite{geman:1980}. For $\delta>1$, Theorem~\ref{thm:MPlaw} implies that the Marchenko-Pastur law holds almost surely, from which it follows that $\liminf \lambda_{\max}(V_n) \geq y_+$ almost surely, so it suffices to show $\limsup \lambda_{\max}(V_n) \leq y_+$ almost surely. To this end, we fix $z > y_+$ and use 
\[
\Prob(\lambda_{\max}(V_n) > z)\ =\ \Prob\left(\left(\frac{\lambda_{\max}(V_n)}{z}\right)^{k_n}>1\right) \leq \E \left(\frac{\lambda_{\max}(V_n)}{z}\right)^{k_n}
\]
and show 
\[
\sum_{n\in\N} \E \left(\frac{\lambda_{\max}(V_n)}{z}\right)^{k_n} < \infty
\]
for a suitable integer sequence $(k_n)_n \to \infty$. By the  Borel--Cantelli lemma, this will imply $\Prob(\limsup\lambda_{\max}(V_n)\leq z) =1$. Writing $k=k_n$ and using
\[
\lambda_{\max}(V_n)^k\ \leq\ \lambda_{1}(n)^k + \ldots + \lambda_{p}(n)^k \ =\ \tr V_n^k,
\]
it suffices to find a sequence $k=k_n$ such that the terms
\begin{equation}
\label{eq:toanalyze}
\frac{1}{z^{k}} \E \tr V_n^{k}\ =\ \frac{1}{z^{k}n^{k}} \sum_{\ubar{s}\in\oneto{p}^{k},\ubar{t}\in\oneto{n}^{k}}\E X_n(\ubar{s},\ubar{t})
\end{equation}
are summable over $n$. To this end, as in \cite{geman:1980} we will choose
\begin{equation}
\label{eq:kndef}
k = \lfloor w\log(n)\rfloor,\quad \text{ where $w>5$ is a constant with }	\quad w\log(y_+/z) < -3.
\end{equation}

Our analysis will greatly depend on the number of singles and doubles present in the paths $(\ubar{s},\ubar{t})$ in \eqref{eq:toanalyze}, and also on the number of distinct nodes that these paths visit. In this respect, the following lemma provides valuable information. For any given path $(\ubar{s},\ubar{t})$ we denote by $r\defeq r(\ubar{s},\ubar{t}) \defeq \# V(\ubar{s})$ the number of rows and by $c\defeq c(\ubar{s},\ubar{t})$ the number of columns that $(\ubar{s},\ubar{t})$ visits, and we set $\ell\defeq r+c$.

\begin{lemma}
\label{lem:pathcomb}
Let $(\ubar{s},\ubar{t})\in \oneto{p}^k\times\oneto{n}^k$ be arbitrary.  Consider $\ell\defeq r+c$ and the product $X_n(\ubar{s},\ubar{t})$ as defined in \eqref{eq:STwalk}. Then the following statements hold:
\begin{enumerate}[i)]
\item $X_n(\ubar{s},\ubar{t})$ has at least $\ell-1$ distinct entries. 
\item $X_n(\ubar{s},\ubar{t})$ has at most $\min(2k,(\ell/2)^2)$ distinct entries.
\item If $X_n(\ubar{s},\ubar{t})$ does not contain singles, then $2\leq \ell\leq k+1$.
\item If $k+2\leq \ell \leq 2k$, then $X_n(\ubar{s},\ubar{t})$ contains at least $2\ell-2k-2 \geq 2$ singles.
\item If $X_n(\ubar{s},\ubar{t})$ has no singles, so $2\leq \ell \leq k+1$, then $X_n(\ubar{s},\ubar{t})$ has at least $3\ell - 2k -3$ doubles.
\item If $X_n(\ubar{s},\ubar{t})$ has $s\geq 1$ singles, then it has at least $3\ell -2s -2k -3$ doubles.
\item If $X_n(\ubar{s},\ubar{t})$ has $s\geq 1$ singles, then $(4\vee 2\sqrt{s})\leq \ell \leq k+s/2$.
\end{enumerate}
\end{lemma}
\begin{proof}
For $i)$, note that after the first entry, there are $r-1$ row and $c-1$ column innovations left. For $ii)$, we observe that a bipartite graph with $r$ nodes on the S-line and $c$ nodes on the T-line has at most $r\cdot c\leq (\ell/2 )^2$ edges. For $iii)$, if $r+c\geq k+2$ were possible, then $X_n(\ubar{s},\ubar{t})$ would have at least $k+1$ entries, which is impossible if each one is to occur at least twice. For $iv)$, note that there are $z\geq \ell -1$ distinct variables. Delete from each a copy, then there are $2k -z$ variables left, none of which were singles. Therefore, $2k-z$ is an upper bound on the variables with multiplicity $\geq 2$, so $z - (2k -z) = 2z -2k$ is a lower bound on the singles. It follows that $2\ell - 2k -2$ is also a lower bound. Statements $v)$ and $vi)$ are proven similarly to $iv)$. For $vii)$, the lower bound $\ell\geq 4$ is elementary. For the other lower bound, note that $X_n(\ubar{s},\ubar{t})$ has at most $(\ell/2)^2$ entries by $ii)$, thus $s \leq (\ell/2)^2$. The upper bound follows with the first statement in Lemma~\ref{lem:nodeandtuplecount} b).
\end{proof}
Returning to \eqref{eq:toanalyze}, we denote by $\alpha_{\ell,s}$ a valid bound for any $\E X_n(\ubar{s},\ubar{t})$, where the path $(\ubar{s},\ubar{t})$ visits $\ell$ nodes and has $s$ singles. Likewise, by $\beta_{\ell,s}$ we denote an upper bound on the number of paths $(\ubar{s},\ubar{t})\in\oneto{p}^k\times\oneto{n}^k$ with $s$ singles and $\ell$ different nodes. We point out that the constants $\alpha_{\ell,s}$ and $\beta_{\ell,s}$ also depend on the length $2k$ of the cycles and on $n$, but we will suppress this dependence in what follows (except in the proof of Lemma~\ref{lem:alphabetabounds}). We can now bound \eqref{eq:toanalyze} by
\begin{equation}
\label{eq:toanalyze2}
\frac{1}{z^{k}n^{k}} \sum_{s=0}^{2k}\sum_{\ell=2}^{2k}\alpha_{\ell,s}\beta_{\ell,s} = \frac{1}{z^{k}n^{k}}\sum_{\ell=2}^{k+1}\alpha_{\ell,0}\beta_{\ell,0} + \frac{1}{z^{k}n^{k}}\sum_{s=1}^{2k}\sum_{\ell=4\vee2\sqrt{s}}^{k+s/2}\alpha_{\ell,s}\beta_{\ell,s},
\end{equation}
where we also used Lemma~\ref{lem:pathcomb}.
The following lemma establishes bounds on the terms $\alpha_{\ell,s}$ and $\beta_{\ell,s}$ in \eqref{eq:toanalyze2}.

\begin{lemma}
\label{lem:alphabetabounds}
If $(\ubar{s},\ubar{t})$ has $0\leq s \leq 2k$ single edges, then for all $n$ large enough (independent of $s$ and $\ell$) we have the following valid bounds on $\alpha_{\ell}$ and $\beta_{\ell}$ for $\ell \in \{2,\ldots,k+1\}$, where $\beta>y_+$ is arbitrary but fixed. 
\begin{enumerate}[a)]
\item $\alpha_{\ell} \leq \#\Pcal\Pcal(2k) = (2k-1)!! = \frac{(2k)!}{2^k k!}\leq k^k$.
\item $\alpha_{\ell} \leq 4(2k)^{(6(k-\ell) + 6 + 4s)\wedge 2k}\frac{1}{n^{\delta s/2}}$.
\item $\beta_{\ell,s}\leq \beta^{\ell} n^{\ell} \ell^{4k}$.
\item $\beta_{\ell,s}\leq (2k)^{2k-s} \binom{2k}{\ell}\beta^{\ell}n^{\ell} \ell^{2k-\ell}$.
\item Whenever $\ell\geq \frac{w-1}{w}(k+s)$, we have 
\begin{equation}
\label{eq:betaellsbound}
\beta_{\ell,s} \leq (1 + \sqrt{p/n})^{2(k+s)} n^{\ell} k^{12(k+s-\ell) + 14}.
\end{equation}
\end{enumerate}
\end{lemma}

We defer the proof of the bounds in Lemma~\ref{lem:alphabetabounds} to Subsection~\ref{sec:boundproof} and take them for granted for now. We would like to see that the r.h.s.\ of \eqref{eq:toanalyze2} is summable over $n$. For the first term, this follows directly from the analysis in \cite{geman:1980}. To see this, note that for $s=0$, the bounds of Lemma~\ref{lem:alphabetabounds} $a)$, $b)$, $c)$ and $e)$ coincide with the bounds of Lemma 1 and Lemma 2 in \cite{geman:1980}, the only difference being a prefactor of $4$ in our bound b), which does not affect the analysis. Therefore, it suffices to establish that 

\begin{equation}
\label{eq:secondsum}
\frac{1}{z^{k}n^{k}}\sum_{s=1}^{2k}\sum_{\ell=4\vee2\sqrt{s}}^{k+s/2}\alpha_{\ell,s}\beta_{\ell,s}
\end{equation}
is summable over $n$. Here, we distinguish between three regimes: The case where $s$ and $\ell$ are both small, the case where $s$ is small and $\ell$ is large, and the case where $s$ is large and $\ell$ is arbitrary. To clarify where we differentiate between ``large" and ``small", we consult Lemma~\ref{lem:alphabetabounds}. We see that bound $e)$ is only meaningful if
\[
\frac{w-1}{w}(k+s) \leq k+s/2 \quad \Leftrightarrow \quad  s\leq \frac{2k}{w-2}.
\]
This is the small $s$ regime, the complement is the large $s$ regime. Further, if $s$ is small, then compatible with $e)$,  $\ell< (w-1)/w(k+s)$ is regarded small, and  $\ell\geq (w-1)/w(k+s)$ is regarded large. We will use the following basic lemma several times in what follows.
\begin{lemma}
\label{lem:termsummable}
For any constants $a,c,d>0$ the term
\[
\frac{(c\log(n))^{d\log(n)}}{n^{a\log(n)}}
\]
is summable over $n$.	
\end{lemma}
\begin{proof}
Taking the $\log$ of the term and dividing by $\log(n)$ yields
\[
d(\log(c) + \log\log(n)) - a\log(n),
\]
which is smaller than $-2$ for all $n$ large enough. Since $\sum_{n=1}^{\infty} n^{-2}$ is finite, the desired result follows.
\end{proof}

\noindent\underline{Case 1: $s$ is small and $\ell$ is large}\newline
In this case, we apply the bounds b) and e) of Lemma~\ref{lem:alphabetabounds} and calculate
\begin{align}
&\frac{1}{z^{k}n^{k}}\sum_{s=1}^{\frac{2k}{w-2}}\sum_{\ell=\frac{w-1}{w}(k+s)}^{k+\frac{s}{2}}\alpha_{\ell,s}\beta_{\ell,s}\notag\\
&\leq\frac{1}{z^{k}n^{k}}\sum_{s=1}^{\frac{2k}{w-2}}\frac{1}{n^{\delta s/2}} \sum_{\ell=\frac{w-1}{w}(k+s)}^{k+\frac{s}{2}}4(2k)^{6(k-\ell)+6+4s}(1 + \sqrt{p/n})^{2(k+s)} n^{\ell} (k+s)^{12(k+s-\ell) + 14}\notag\\
&\leq 4\left(\frac{(1 + \sqrt{p/n})^2}{z}\right)^k(2k)^{20} \sum_{s=1}^{\frac{2k}{w-2}}\frac{\left[(2k)^4 (1+\sqrt{p/n})^2 \right]^s}{n^{\delta s/2}}(k+s)^{12s}\sum_{\ell = \frac{w-1}{w}(k+s)}^{k+\frac{s}{2}}\left(\frac{(2k)^6(k+s)^{12}}{n}\right)^{k-\ell}\notag\\
&\leq 4\left(\frac{(1 + \sqrt{p/n})^2}{z}\right)^k(2k)^{20} \sum_{s=1}^{\frac{2k}{w-2}}\frac{\left[(2k)^4 (1+\sqrt{p/n})^2 \right]^s}{n^{\delta s/2}}(k+s)^{12s+1}\left(\frac{n}{(2k)^6(k+s)^{12}}\right)^{\frac{s}{2}}\notag\\
&= 4\left(\frac{(1 + \sqrt{p/n})^2}{z}\right)^k(2k)^{20} \sum_{s=1}^{\frac{2k}{w-2}}\frac{\left[2k (1+\sqrt{p/n})^2 \right]^s}{n^{(\delta-1) s/2}}(k+s)^{6s+1}\notag\\
&\leq 4\left(\frac{(1 + \sqrt{p/n})^2}{z}\right)^k(2k)^{21} \sum_{s=1}^{\frac{2k}{w-2}}\frac{\left[(2k)^7 (1+\sqrt{p/n})^2 \right]^s}{n^{(\delta-1) s/2}}\label{eq:smalllarge}
\end{align}
which is summable over $n$, since the last sum remains bounded by some constant $C$. Then, with $z > (1+\sqrt{y})^2$, $p/n\to y$ and \eqref{eq:kndef}, we find for any constants $c,d>0$ that for all $n$ large enough,
\[
\frac{\log\left[\left(\frac{(1 + \sqrt{p/n})^2}{z}\right)^k c\log(n)^d\right]}{\log(n)} \leq w\log\left(\frac{(1 + \sqrt{p/n})^2}{z}\right) + \frac{\log(c)}{\log(n)} + \frac{d\log\log(n)}{\log(n)} \leq -2,
\]
which shows summability of \eqref{eq:smalllarge}. Note that it was used that $s$ is small, since only for these small $s$ we may use the bound Lemma~\ref{lem:alphabetabounds} $e)$ for large $\ell$. \newline
\underline{Case 2: $s$ and $\ell$ are both small.}\newline
In this case, we employ the bounds $b)$ and $c)$ of Lemma~\ref{lem:alphabetabounds} and calculate

\begin{align}
&\frac{1}{z^{k}n^{k}}\sum_{s=1}^{\frac{2k}{w-2}}\sum_{\ell=4\vee 2\sqrt{s}}^{\frac{w-1}{w}(k+s)}\alpha_{\ell,s}\beta_{\ell,s}\ \leq\  \frac{1}{z^{k}n^{k}}\sum_{s=1}^{\frac{2k}{w-2}}\frac{1}{n^{\delta s/2}}\sum_{\ell=4\vee 2\sqrt{s}}^{\frac{w-1}{w}(k+s)}4(2k)^{2k}\beta^{\ell}n^{\ell} \ell^{4k}
\notag\\
&\leq \frac{4(2k)^{6k}\beta^{2k}}{z^{k}n^{k}}\sum_{s=1}^{\frac{2k}{w-2}}\frac{1}{n^{\delta s/2}}\sum_{\ell=4\vee 2\sqrt{s}}^{\frac{w-1}{w}(k+s)}n^{\ell}\
\leq \ \frac{4(2k)^{6k+1}\beta^{2k}}{z^{k}n^{k}}\sum_{s=1}^{\frac{2k}{w-2}}\frac{1}{n^{\delta s/2}}n^{\frac{w-1}{w}(k+s)}\notag\\
&\leq \frac{4(2k)^{6k+1}\beta^{2k}}{z^{k}n^{\frac{1}{w}k}}\sum_{s=1}^{\frac{2k}{w-2}}\frac{1}{n^{\left(\frac{\delta}{2} - \frac{w-1}{w}\right) s}}. \label{eq:case2}
\end{align}
To see that \eqref{eq:case2} is summable for any choice of $\delta>1$, we distinguish two cases: First, if $\delta/2 > (w-1)/w$, the sum in \eqref{eq:case2} remains bounded as $n\to\infty$ and summability of the entire term follows with Lemma~\ref{lem:termsummable}. On the other hand, if $\delta>1$ but $\delta/2\leq (w-1)/w$, the term in \eqref{eq:case2} is bounded by
\[
\frac{4(2k)^{6k+2}\beta^{2k}}{z^{k}n^{\frac{1}{w}k}}\cdot n^{\left(\frac{w-1}{w}-\frac{\delta}{2}\right) \frac{2k}{w-2}} = \frac{4(2k)^{6k+2}\beta^{2k}}{z^{k}n^{\frac{\delta-1}{w-2}k}},
\]
which is also summable with Lemma~\ref{lem:termsummable}.
\newline
\underline{Case 3: $s$ is large.}\newline
In this case, we apply the bounds $b)$ and $c)$ and calculate
\begin{align*}
&\frac{1}{z^{k}n^{k}}\sum_{s=\frac{2k}{w-2}}^{2k}\sum_{\ell=4\vee 2\sqrt{s}}^{k+s/2}\alpha_{\ell,s}\beta_{\ell,s}\ \leq\ \frac{1}{z^{k}n^{k}}\sum_{s=\frac{2k}{w-2}}^{2k}\sum_{\ell=4\vee 2\sqrt{s}}^{k+s/2}4(2k)^{2k}\frac{1}{n^{\delta s}}\beta^{\ell}n^{\ell}\ell^{4k}\\
&\leq \frac{2(2k)^{6k}\beta^{2k}}{z^k}\sum_{s=\frac{2k}{w-2}}^{2k}\sum_{\ell=4\vee 2\sqrt{s}}^{k+s/2}\frac{1}{n^k}\frac{1}{n^{\delta s}}n^{\ell}\ \leq\ \frac{2(2k)^{6k+1}\beta^{2k}}{z^k}\sum_{s=\frac{2k}{w-2}}^{2k}\frac{1}{n^{(\delta-1) s}}\\
&\leq \frac{2(2k)^{6k+2}\beta^{2k}}{z^k}\frac{1}{n^{(\delta-1) 2k/(w-2)}}\\
\end{align*}
which is summable by Lemma~\ref{lem:termsummable}.

\subsection{Proof of Lemma~\ref{lem:alphabetabounds}}
\label{sec:boundproof}
We recall that 
\begin{equation}
\label{eq:sumprod}
\E X_n(\ubar{s},\ubar{t}) = \sum_{\pi\in\Pcal\Pcal(2k)}\prod_{\{u,v\}\in\pi} \Sigma_n\left(e_u(\ubar{s},\ubar{t}),e_v(\ubar{s},\ubar{t})\right),	
\end{equation}
so a) follows immediately, since each factor in the product is bounded by $1$. Statement c) is trivial and follows from (the proof of) Lemma 2 (a) in \cite{geman:1980}. For d), we first allocate the non-single edges to $2k$ places, for which we have at most $(2k)^{2k-s}$ possibilities. This also determines the places of the singles. Then we pick the nodes which may be chosen freely, for which we have at most $\binom{2k}{\ell}$ possibilities. For the nodes that may be chosen freely, we have at most $\beta^{\ell}n^{\ell}$ choices. For all other nodes we have at most $\ell^{2k-\ell}$ choices. For $e)$, the statement is well-known for $s=0$, see \cite{geman:1980} Lemma 2 (b). If $(\ubar{s},\ubar{t})$ has $s\geq 1$ singles, repeat each single twice after its first appearance, yielding a unique path of length $2(k+s)$ without singles and with $\ell$ nodes. In formulas, we observe $\beta^{(2k)}_{\ell,s}\leq \beta^{(2(k+s))}_{\ell,0}$. 
Now the following statement holds: For all $n$ large enough we find that for all $s\geq 0$ and $\ell \geq ((w-1)/w) (k+s)$, $\beta^{(2(k+s))}_{\ell,0}$ satisfies the bound in \eqref{eq:betaellsbound}. This is readily seen by repeating the proof of Lemma 2 b) in \cite{geman:1980} after replacing $k$ with $k+s$ at each step. The only place where heed must be taken is equation (13) in \cite{geman:1980} where it must be argued that for all $n$ large enough, we find \emph{for all $s\geq 0$} that $2(k+s-\ell) + 2 \leq (\ell -1)/2$ whenever $\ell\geq ((w-1)/w)(k+s)$. But this follows since 
\[
2(k+s-\ell) + 2 \leq \frac{\ell -1}{2}\ \Leftrightarrow\ \frac{4}{5}(k+s)+1 \leq \ell
\]
and using that $w>5$ and $k\to\infty$, we find
\[
\ell\geq \frac{w-1}{w}(k+s)\ \Rightarrow \ \ell\geq\frac{4}{5}(k+s)+1
\]
for all $n$ large enough such that
\[
\frac{w-1}{w}k \geq \frac{4}{5}k +1.
\]
Note that this choice of $n$ is independent of $s$.

It remains to show b).
For b), we fix a path $(\ubar{s},\ubar{t})$ with $s$ single edges and with $r+c =\ell$. Denote by $d$ the number of distinct double edges in the path. By Lemma~\ref{lem:pathcomb} we know that 
\begin{equation}
\label{eq:dlower}
d\geq 3\ell-2s-2k-3.
\end{equation}
To bound \eqref{eq:sumprod}, we first consider the pair partitions in \eqref{eq:sumprod} which are contained in the sense that $\pi=(\pi',\pi'')$ where $\pi'$ is a pairing of the $2d$ variables from the $d$ doubles and $\pi''$ is a pairing of the remaining $2k-2d$ non-double variables, out of which $s$ are single. 

For the non-double variables, we get $\#\Pcal\Pcal(2k-2d)$ possible pairings, each leading to a contribution of at most $n^{\delta s/2}$ in the product over their blocks, since in the worst case, all singles are paired with each other, and the edges of multiplicity $\geq 3$ are also paired with each other.

For the doubles, we have exactly one partition $\pi^*\in\Pcal\Pcal(2d)$ which pairs all the doubles, leading to a contribution of exactly $1$ in the partial product over its blocks in \eqref{eq:sumprod}. All other $\pi'\in\Pcal\Pcal(2d)$ have $q\in\{2,\ldots,d\}$ void pairings (a void paring is a block which does not pair the two edges of a double edge) which then lead to a factor of at most $1/n^{q\delta}$ in the partial product over its blocks in \eqref{eq:sumprod}. Define for $q\in\{2,\ldots,d\}$,
\[
\Pcal\Pcal(2d)_q \defeq \{\pi'\in\Pcal\Pcal(2d)\,|\, \pi' \text{ has $q$ void pairings} \}.
\]
Note that each $\pi=(\pi',\pi'')\in\Pcal\Pcal(2k)$ with $\pi'\in\Pcal\Pcal(2d)_q$ yields a contribution of at most $1/n^{q\delta}$ in the sum \eqref{eq:sumprod}.

To construct a $\pi'\in\Pcal\Pcal(2d)_q$, we first have $\binom{d}{q}$ choices for the doubles which shall not be paired. Then we choose a pair partition $\vartheta\in\Pcal\Pcal(2q)$ which only has void pairings, for which we have at most $\#\Pcal\Pcal(2q)$ possibilities. For all other $d-q$ pairs there is only one pair partition in $\Pcal\Pcal(2(d-q))$ which pairs the pairs, so
\[
\#\Pcal\Pcal(2d)_q \leq d^q\cdot q^q \leq d^{2q}.
\]
As a result, for the sum in \eqref{eq:sumprod} over the contained pair partitions we obtain
\begin{align}
&\sum_{\substack{\pi\in\Pcal\Pcal(2k) \\ \text{$\pi$ contained}}}\prod_{\{u,v\}\in\pi} \abs{\Sigma_n\left(e_u(\ubar{s},\ubar{t}),e_v(\ubar{s},\ubar{t})\right)} \ \leq\ \frac{1}{n^{\delta s/2}}\#\Pcal\Pcal(2k-2d)\cdot \left(\sum_{q=0,2,3,\ldots}^d \frac{d^{2q}}{n^{q\delta}}\right)\notag \\
&\leq \frac{1}{n^{\delta s/2}}\#\Pcal\Pcal(2k-2d) \left(1 + \frac{d^5}{n^{2\delta}}\right) \ \leq\ 2(2k)^{2k-2d}\frac{1}{n^{\delta s/2}}
\label{eq:SinglesContained}
\end{align}
for all $n$ large enough. This finishes the analysis of contained partitions $\pi\in\Pcal\Pcal(2k)$.

It is left to analyze the contribution of the non-contained $\pi\in\Pcal\Pcal(2k)$ in \eqref{eq:sumprod}, that is, those pair partitions which pair a double variable with a non-double variable. Each such $\pi$ has
\begin{itemize}
\item at least $2$ traversings (i.e. blocks that pair a double variable with a non-double variable),
\item at most $2d\wedge (2k-2d)$ traversings,
\item always an even number of traversings.
\end{itemize}
Each traversing block yields a factor of decay of $1/n^{\delta}$. Let $t\defeq \#\text{traversings}/2$, $2t\in\{2,4,\ldots,(2d)\wedge(2k-2d)\}$. Then to determine the $2t$ traversing blocks, we first pick $2t$ out of the $2d$ double elements, and have for each at most $(2k-2d)$ choices for a partner among the $2k-2d$ variables, leading to at most $\binom{2d}{2t}(2k-2d)^{2t}$ choices for the traversing blocks. All other blocks contain either double variables only or non-double variables only. The $2k-2d-2t$ non-double variables contain at least $(s-2t)\vee 0$ singles, so we get a decay of at least $1/n^{\delta(s/2-t)}$, for each of the possible $\Pcal\Pcal(2k-2d-2t)$ pairings of these elements.

For the blocks over the remaining doubles we may have $v\in\{0,1,2,\ldots,d-t\}$ void pairings of the doubles, each leading to a decay of $1/n^{\delta v}$. To complete the pair partition, we need to determine the blocks for the remaining $2d-2t$ double elements in such a way that we get exactly $v\in\{0,1,2,\ldots,d-t\}$ void pairings. Due to the configuration of the traversing blocks, this might not be possible for a given $v$ ($0$ possibilities), but we can always construct an upper bound: So we assume that it is possible to have exactly $v$ void pairings. This leaves $2v$ double elements to be paired in a void way for which we have at most $\#\Pcal\Pcal(2v)$ choices. Since the remaining double elements are paired, this leaves only one choice for these blocks. Combining everything we just argued, we obtain

\begin{align}
&\sum_{\substack{\pi\in\Pcal\Pcal(2k) \\ \text{$\pi$ non-contained}}}\prod_{\{u,v\}\in\pi} \abs{\Sigma_n\left(e_u(\ubar{s},\ubar{t}),e_v(\ubar{s},\ubar{t})\right)} \notag\\
&\leq
 \sum_{t=1}^{d\wedge(k-d)} \frac{1}{n^{\delta 2t}}\binom{2d}{2t}(2k-2d)^{2t} 
 \sum_{v=0}^{d-t}\frac{1}{n^{v\delta}}\binom{d-t}{v}\#\Pcal\Pcal(2v)\frac{1}{n^{\delta(s/2-t)\vee 0}}\#\Pcal\Pcal(2k-2d-2t)\notag\\
 &\leq 
(2k)^{2k-2d} \sum_{t=1}^{d\wedge(k-d)} \underbrace{\frac{1}{n^{\delta 2t}}\frac{1}{n^{\delta(s/2-t)\vee 0}}}_{\leq \frac{1}{n^{\delta s/2}}\frac{1}{n^{\delta t}}}(2d)^{2t}
 \underbrace{\sum_{v=0}^{d-t}\frac{1}{n^{v\delta}}(d-t)^{2v}}_{\leq 2 \text{ for $n$ large enough}} \frac{1}{n^{\delta(s/2-t)\vee 0}}\notag\\
 &\leq 2(2k)^{2k-2d} \frac{1}{n^{\delta s/2}}\underbrace{\sum_{t=1}^{d\wedge(k-d)} \left(\frac{(2d)^2}{n^{\delta}}\right)^t}_{\leq 1 \text{ for $n$ large enough}}.\label{eq:SinglesNonContained}
 \end{align}
 Combining \eqref{eq:SinglesNonContained} and \eqref{eq:SinglesContained} together with \eqref{eq:dlower} yields
 \[
 \alpha_{\ell} \leq 4(2k)^{2k-2d}\frac{1}{n^{\delta s /2}} =  4(2k)^{(2k-2d)\wedge 2k}\frac{1}{n^{\delta s /2}} \leq 4(2k)^{(6(k-\ell) + 6 + 4s)\wedge 2k}\frac{1}{n^{\delta s/2}},
 \]
which is the statement of $b)$.

\subsection{Proof of Theorem~\ref{thm:rankoneexample}}
\label{sec:divergence}
Let $\delta>0$ and let for all $n\in\N$ and all $(i,j)\in\oneto{p}\times\oneto{n}$:
\begin{equation}
\label{eq:rankoneexample}
X_n(i,j) \defeq \sqrt{1-\frac{1}{n^{\delta}}} Y_n(i,j) + \sqrt{\frac{1}{n^{\delta}}}Z,
\end{equation}
where $Y_n$ is a $p\times n$ matrix with i.i.d.\ $\Ncal(0,1)$ distributed entries, $Z$ is independent of $\{Y_n(i,j)\,|\, i\in \oneto{p}, j\in\oneto{n}, n\in\N\}$ and also $\Ncal(0,1)$ distributed. Then $\E X_n(i,j)=0$, $\V X_n(i,j)=1$ and for $(i,j)\neq(k,l)$ we have $\E X_n(i,j)X_n(k,l) = 1/n^{\delta}$.

Let $\delta>0$ and denote by $(Z)_{a\times b}$ the $a\times b$ matrix with entry $Z$ at each place. Then
\begin{align*}
V_n &= \frac{1}{n}Y_nY_n^T + \frac{1}{n^{1+\delta}}Y_nY_n^T \notag\\
&\quad +\frac{1}{n}\sqrt{\left(1-\frac{1}{n^{\delta}}\right)\left(\frac{1}{n^{\delta}}\right)}\big(Y_n(Z)_{n\times p}+(Z)_{p\times n}Y_n^T\big) + \frac{1}{n^{\delta+1}}(Z)_{p\times n}(Z)_{n\times p} \notag\\
&=: A_n+B_n+C_n+D_n\,.
\end{align*}
Since the entries in $Y_n$ are i.i.d.\ Gaussian, it is well known that $\lim_{\nto} \opnorm{A_n}\to y_+$ almost surely, which also implies that $\lim_{\nto} \opnorm{B_n}\to 0$ almost surely.
Now we turn to $D_n$ and see that
\[
D_n = \frac{1}{n^{\delta+1}}(nZ^2)_{p\times p}=\frac{1}{n^{\delta}}(Z^2)_{p\times p}
\]
has eigenvalues $n^{-\delta} p Z^2$ and $0$, from which we deduce 
\[
\frac{n^{\delta}}{p} \opnorm{D_n}\ \xrightarrow[a.s.]{}\ 
Z^2\,, \qquad \nto\,,
\]
so in particular $n^{\delta-1}\opnorm{D_n}\to y Z^2$ almost surely.
Next, we will show that $C_n$ is negligble. To this end, we calculate
\[
Y_n(Z)_{n\times p}  = 
\begin{pmatrix}
\sum_{i=1}^n Y(1,i)Z &\ldots & \sum_{i=1}^n Y(1,i)Z\\
\sum_{i=1}^n Y(2,i)Z &\ldots & \sum_{i=1}^n Y(2,i)Z\\
\vdots &\ldots & \vdots\\
\sum_{i=1}^n Y(p,i)Z &\ldots & \sum_{i=1}^n Y(p,i)Z
\end{pmatrix},
\]
the eigenvalues of which are $0$ and the sum of the entries in the first column. Therefore,
\begin{align*}
&\opnorm{n^{-1}\sqrt{(1-n^{-\delta})n^{-\delta}}Y_n(Z)_{n\times p}}\\
& =\  |Z| \frac{1}{n}\sqrt{\left(1-\frac{1}{n^{\delta}}\right)\left(\frac{1}{n^{\delta}}\right)}\bigg|\sum_{(s,t)\in\oneto{p}\times\oneto{n}}Y_n(s,t)\bigg|\\
&\quad \sim |Z|\frac{1}{n^{1+\delta/2}}\bigg|\sum_{(s,t)\in\oneto{p}\times\oneto{n}}Y_n(s,t)\bigg|\\
&\quad \sim |Z|\frac{y^{1/2+\delta/4}}{(np)^{1/2+\delta/4}}\bigg|\sum_{(s,t)\in\oneto{p}\times\oneto{n}}Y_n(s,t)\bigg|\xrightarrow[n\to\infty]{} 0
\end{align*}
almost surely for any $\delta>0$, which follows from the LLN of Marcinkiewicz-Zygmund (see e.g.\ \cite{Durrett}). Similarly it follows that $\opnorm{n^{-1}\sqrt{(1-n^{-\delta})n^{-\delta}}(Z)_{p\times n}Y_n^T}\to 0$ almost surely, so that
\begin{equation}
\label{eq:ADsuffice}
\abs{\opnorm{V_n}- \opnorm{A_n+D_n}} \leq \opnorm{B_n} + \opnorm{C_n} \xrightarrow[n\to\infty]{} 0
\end{equation}
almost surely. Since $\opnorm{A_n} - \opnorm{D_n}\leq \opnorm{A_n + D_n} \leq \opnorm{A_n} + \opnorm{D_n}$ respectively $\opnorm{D_n}-\opnorm{A_n} \leq \opnorm{A_n + D_n}$, our findings immediately imply statements $i)$ respectively $ii)$ (of course, statement $i)$ also follows directly with part $i)$ of Theorem~\ref{thm:opnorm}).

Finally, we turn to the case where $\delta=1$ which is more involved because $\opnorm{A_n}$ and $\opnorm{D_n}$ are both of constant order. We have
\begin{equation*}
A_n+D_n= \frac{1}{n} \Big( Y_nY_n^T + Z^2 (1)_{p\times p}\Big)\,,
\end{equation*}
where $(1)_{p\times p}$ denotes the $p\times p$ matrix with all entries one. In order to analyze the largest eigenvalue of $A_n+D_n$, we note that for any orthogonal matrix $\mathcal{O}$ it holds
\begin{equation}\label{eq:usefuldd}
\lambda_{\max}(A_n+D_n)= \lambda_{\max}\big( \mathcal{O}(A_n+D_n)\mathcal{O}^T\big)\,.
\end{equation}
For any real symmetric matrix $M$, we denote
its spectral decomposition by
$M=\mathcal{O}_{M} \Lambda_{M} \mathcal{O}_{M}^{T}$,
where $ \Lambda_{M}$ is the diagonal matrix whose $i$-th diagonal element is the $i$-th largest eigenvalue of $M$, and $\mathcal{O}_{M}$ is an orthogonal matrix. Choosing $\mathcal{O}=\mathcal{O}_{(1)_{p\times p}}$ in \eqref{eq:usefuldd}, we get
\begin{equation*}
\lambda_{\max}(A_n+D_n)= \lambda_{\max}\bigg( 
\mathcal{O}_{(1)_{p\times p}}\Big(\frac{1}{n}  Y_nY_n^T\Big) \mathcal{O}_{(1)_{p\times p}}^T +Z^2 \frac{p}{n} (1,0)_{p\times p} \bigg)\,,
\end{equation*}
where $(1,0)_{p\times p}$ is the $p\times p$ matrix with $(1,1)$-entry equal to $1$ and all other entries zero. Since the matrix $Y_n$ is orthogonally invariant, that is $Y_n \eid \mathcal{O} Y_n$ for any orthogonal matrix $\mathcal{O}$ independent of $Y_n$, we conclude
\begin{equation}\label{eq:usefuld2}
\lambda_{\max}(A_n+D_n)\eid \lambda_{\max}\bigg( 
\frac{1}{n}  Y_nY_n^T +Z^2 \frac{p}{n} (1,0)_{p\times p} \bigg)\,.
\end{equation}
In order to study the right-hand side in \eqref{eq:usefuld2}, we introduce for $\eta\in \R$ the $(n+1)\times (n+1)$ diagonal matrices $Q$ with $Q(n+1,n+1)=\eta$ and $Q(j,j)=1$, $j=1, \ldots,n$. Further let $M$ be an $(n+1)\times p$ matrix of i.i.d.\ standard Gaussian random variables. Then it holds that $QMM^TQ$ and $M^TQ^2M$ have the same set of nonzero eigenvalues, and in particular $\lambda_{\max}(QMM^TQ)=\lambda_{\max}(M^TQ^2M)$.
By either \cite[Theorems 1.1-1.3]{baik:silverstein:2006} or \cite[Theorems 4.1 and 4.2]{bai:yao:2012} (compare also with Theorem 11.3 in \cite{yao:zheng:bai:2015}), 
we have, 
\begin{equation}\label{eq:A}
\frac{1}{p} \lambda_{\max}(QMM^TQ) \xrightarrow[n\to\infty]{\text{a.s.}}
\left\{\begin{array}{ll}
\left(1+\sqrt{y^{-1}}\right)^2 \,, & \mbox{if } \eta^2\le 1+ \sqrt{y^{-1}}, \\
\eta^2 \left(1+\frac{y^{-1}}{\eta^2-1}\right) \,, & \mbox{if } \eta^2> 1+ \sqrt{y^{-1}}.
\end{array}\right.
\end{equation}
Writing $m_1, \ldots, m_{n+1} \in \R^p$ for the columns of $M^T$ and using the notation $\mathcal{O}_{n+1}=\mathcal{O}_{m_{n+1} m_{n+1}^T}$, we get
\begin{align*}
\frac{1}{n}\lambda_{\max}(M^TQ^2M)&= \frac{1}{n}\lambda_{\max}\left( \sum_{i=1}^n m_i m_i^T +\eta^2 m_{n+1} m_{n+1}^T \right)\\
&= \frac{1}{n}\lambda_{\max}\left( \mathcal{O}_{n+1} \Big(\sum_{i=1}^n m_i m_i^T\Big) \mathcal{O}_{n+1}^T  +\eta^2  m_{n+1}^T m_{n+1} (1,0)_{p\times p} \right)\\
&= \frac{1}{n}\lambda_{\max}\left( \mathcal{O}_{n+1} \Big(\sum_{i=1}^n m_i m_i^T\Big) \mathcal{O}_{n+1}^T  +\eta^2  p (1,0)_{p\times p} \right) + o(1)\,,
\end{align*}
where in the last step we used that $n^{-1} (m_{n+1}^T m_{n+1} -p) \cas 0$ by the law of large numbers. Utilizing the rotational invariance of the normal distribution and the fact that $m_1,\ldots, m_n$ are independent from $m_{n+1}$, we conclude that 
$$\mathcal{O}_{n+1} \Big(\sum_{i=1}^n m_i m_i^T\Big) \mathcal{O}_{n+1}^T\eid \sum_{i=1}^n m_i m_i^T\,.$$
Together with \eqref{eq:A}, this implies, as $\nto$,
\begin{equation}\label{eq:B}
\lambda_{\max}\left( \frac{1}{n} \sum_{i=1}^n m_i m_i^T  +\eta^2  \frac{p}{n} (1,0)_{p\times p} \right) \cas
\left\{\begin{array}{ll}
\left(1+\sqrt{y}\right)^2 \,, & \mbox{if } \eta^2\le 1+ \sqrt{y^{-1}}, \\
\eta^2 \left(y+\frac{1}{\eta^2-1}\right) \,, & \mbox{if } \eta^2> 1+ \sqrt{y^{-1}}.
\end{array}\right.  
\end{equation}
Noting that $Y_nY_n^T\eid \sum_{i=1}^n m_i m_i^T$, we now deduce from \eqref{eq:usefuld2} and \eqref{eq:B} that, since $Y_nY_n^T$ is independent of $Z$,
\begin{equation}\label{eq:C}
\lambda_{\max}(A_n+D_n) \xrightarrow[n\to\infty]{\text{a.s.}}
\left\{\begin{array}{ll}
\left(1+\sqrt{y}\right)^2 \,, & \mbox{if } Z^2\le 1+ \sqrt{y^{-1}}, \\
Z^2 \left(y+\frac{1}{Z^2-1}\right) \,, & \mbox{if } Z^2> 1+ \sqrt{y^{-1}}.
\end{array}\right.  
\end{equation}
Thus, we obtain 
\begin{equation*}
\lambda_{\max}(A_n+D_n)\xrightarrow[n\to\infty]{\text{a.s.}} \left(1+\sqrt{y}\right)^2 \1_{\{Z^2\le 1+ \sqrt{y^{-1}}  \}} + Z^2 \left(y+\frac{1}{Z^2-1}\right) \1_{\{Z^2> 1+ \sqrt{y^{-1}}  \}}\,,
\end{equation*}
completing the proof.

\sloppy
\printbibliography

\vspace{1cm}

\noindent\textsf{(Michael Fleermann)\newline
FernUniversit\"at in Hagen\newline
Fakult\"at f\"ur Mathematik und Informatik\newline 
Universit\"atsstra\ss e 1\newline 
58084 Hagen}\newline
\textit{E-mail address:}
\texttt{michael.fleermann@fernuni-hagen.de}
\vspace{1cm}

\noindent\textsf{(Johannes Heiny)\newline
Ruhr-Universit\"at Bochum\newline
Fakult\"at f\"ur Mathematik\newline 
Universit\"atsstra\ss e 150\newline 
44801 Bochum}\newline
\textit{E-mail address:}
\texttt{johannes.heiny@rub.de}

\end{document}